\documentclass[leqno,a4paper]{article}
\usepackage{amsmath,amsthm,amssymb}
\usepackage[utf8]{inputenc}
\usepackage[T1]{fontenc}

\usepackage{enumerate}
\usepackage{bbm}
\usepackage{todonotes}
\usepackage{mathrsfs}

\usepackage{todonotes}
\usepackage{hyperref}

\usepackage{etoolbox}% http://ctan.org/pkg/etoolbox
\AtBeginEnvironment{align}{\setcounter{subeqn}{0}}% Reset subequation number at start of align
\newcounter{subeqn} \renewcommand{\thesubeqn}{\theequation\alph{subeqn}}%
\newcommand{\subeqn}{%
  \refstepcounter{subeqn}% Step subequation number
  \tag{\thesubeqn}% Label equation
}

\numberwithin{dummy}{subsection}
\numberwithin{equation}{section}

\newtheorem{Th}{Theorem}[subsection]
\newtheorem{Prop}[Th]{Proposition}
\newtheorem{Lemma}[Th]{Lemma}
\newtheorem{Cor}[Th]{Corollary}
\theoremstyle{definition}
\newtheorem{Remark}[Th]{Remark}

\newtheorem{Example}{Example}
\newtheorem{Question}{Question}

\newcommand{\beq}{\begin{equation}}
\newcommand{\eeq}{\end{equation}}

\def\scalar(#1,#2){(#1\mid#2)}

\renewcommand{\hat}{\widehat}

\newcommand{\ca}{{\cal A}}
\newcommand{\cb}{{\cal B}}
\newcommand{\cc}{{\cal C}}

\newcommand{\cf}{{\cal F}}
\newcommand{\cg}{{\cal G}}
\newcommand{\cp}{{\cal P}}

\newcommand{\ch}{{\cal H}}

\newcommand{\xbm}{(X,{\cal B},\mu)}

\newcommand{\ycn}{(Y,{\cal C},\nu)}
\newcommand{\ot}{\otimes}
\newcommand{\ov}{\overline}
\newcommand{\la}{\lambda}

\newcommand{\Q}{\mathbb{Q}}
\newcommand{\R}{{\mathbb{R}}}
\newcommand{\T}{{\mathbb{T}}}

\newcommand{\Z}{{\mathbb{Z}}}
\newcommand{\N}{{\mathbb{N}}}
\newcommand{\PP}{{\mathbb{P}}}
\newcommand{\E}{{\mathbb{E}}}
\newcommand{\vep}{\varepsilon}
\newcommand{\va}{\varphi}

\newcommand{\raz}{\mathbbm{1}}

\newcommand{\sa}{{\underline{s},\underline{a}}}

\usepackage{nicefrac}

\title{On invariant measures for $\mathscr{B}$-free systems}
\author{Joanna Ku\l aga-Przymus\footnotemark[1] \and Mariusz Lema\'nczyk\thanks{Research supported by Narodowe Centrum Nauki grant DEC-2011/03/B/ST1/00407.} \and  Benjamin Weiss}

\begin{document}

\maketitle \normalsize

\thispagestyle{empty}

\begin{abstract}
We show that the $\mathscr{B}$-free subshift $(S,X_{\mathscr{B}})$ associated to a $\mathscr{B}$-free system is intrinsically ergodic, i.e.\ it has exactly one measure of maximal entropy. Moreover, we study invariant measures for such systems.
It is proved that each ergodic invariant measure is of joining type, determined by a joining of the Mirsky measure of a $\mathscr{B}'$-free subshift contained in $(S,X_{\mathscr{B}})$ and an ergodic invariant measure of the full shift on $\{0,1\}^{\Z}$. Moreover, each ergodic joining type measure yields a measure-theoretic dynamical system with infinite rational part of the spectrum corresponding to the above Mirsky measure. Finally, we show that, in general, hereditary systems may not be intrinsically ergodic.
\end{abstract}

\tableofcontents

\section*{Introduction}
Assume that ${\mathscr{B}}=\{b_1,b_2,\ldots\}\subset \{2,3,\ldots\}$ is such that
\beq\label{f1}
(b_i,b_j)=1\text{ whenever } i\neq j\text{ and }\sum_{i\geq1}1/b_i<+\infty.
\eeq
For example, we can take ${\mathscr{B}}=\{p_i^2: i\geq1\}$, where $p_i\in\mathscr{P}$ stands for the $i$th prime number. To $\mathscr{B}$ we associate a two-sided sequence $\eta\in\{0,1\}^{\Z}$ by setting
$$
\eta(n):=\begin{cases}
1& \text{if }b_i\nmid n \text{ for all }i\geq 1,\\
0& \text{otherwise.}
\end{cases}
$$
Then let
$$
X_\eta:=\{y\in\{0,1\}^\Z : \text{each block occurring on }y \text{ occurs on }\eta\}.
$$
We will also write $X_\mathscr{B}$ instead of $X_{\eta}$. Let $S$ stand for the shift transformation on $\{0,1\}^\Z$ and notice that $X_\eta$ is closed and $S$-invariant (shortly, $X_\eta$ is a subshift). We will call $X_\eta$ the $\mathscr{B}${\em-free subshift}. When $b_i=p_i^2$, $i\geq1$, the corresponding subshift is called the {\em square-free}.

When $\mathscr{B}$ satisfies~\eqref{f1}, it follows by~\cite{Ab-Le-Ru} that the topological entropy $h_{top}(S,X_\eta)$ of the subshift $(S,X_\eta)$ is positive. A natural question arises whether there is only one invariant measure~$\nu$ whose entropy $h_\nu(S,X_\eta)$ attains the value of topological entropy, i.e.\ whether $(S,X_\eta)$ is intrinsically ergodic \cite{We}. In Section~\ref{section1}, we will give a simple proof of the following result which for the square-free subshift has been obtained by Peckner in \cite{Pec}.

\begin{Th}\label{benji} For each $\mathscr{B}\subset\N$ satisfying~\eqref{f1}, the corresponding $\mathscr{B}$-free subshift $(S,X_\eta)$ is intrinsically ergodic.\end{Th}

The $\mathscr{B}$-free subshifts turn out to be hereditary systems \cite{Ke-Li}, i.e.\ they have the following property:
\begin{equation}\label{her}
\mbox{whenever $x\in X_\eta$, $y\in \{0,1\}^\Z$ and $y\leq x$ (coordinatewise) then $y\in X_\eta$}.
\end{equation}
In Section~\ref{section3}, we show how to adapt the method used to prove Theorem~\ref{benji} to obtain that other natural hereditary systems are intrinsically ergodic (e.g.\ Sturmian hereditary systems).

In Section~\ref{section2}, we study the set $\cp(S,X_\eta)$ of invariant measures for $(S,X_\eta)$ which is completely determined by the subset $\cp^e(S,X_\eta)$ of ergodic measures.
Among them, the most natural non-trivial member of $\cp^e(S,X_\eta)$ is the so called Mirsky measure $\nu_{\mathscr{B}}$ (see Section~\ref{aaa}) which yields the (ergodic) dynamical system with purely discrete spectrum whose group of eigenvalues consists of all roots of unity of degree $b_1\cdot\ldots\cdot b_k$, $k\geq1$, see \cite{Ab-Le-Ru}, \cite{Ce-Si}, \cite{Sa}. A basic observation (see Proposition~\ref{p-jm}) is that whenever $\rho\neq\delta_{(\ldots,0,0,\ldots)}$ is an ergodic  measure for $(S,X_\eta)$ then the corresponding measure-theoretic system has infinite rational discrete spectrum generated by $b'_k$-roots of unity for some 
\beq
\label{in1} 1<b'_k|b_k \;\mbox{for each}\;k\geq1.
\eeq
As $(S,X_\eta)$ may contain, as a subsystem, another $\mathscr{B}'$-free subshift, the measure-theoretic dynamical system $(S,X_\eta, \nu_{\mathscr{B}'})$  may have essentially smaller spectrum than that determined by $\nu_{\mathscr{B}}$. A natural question arises whether all sequences of  $(b'_k)$ satisfying~\eqref{in1} are ``realizable''. We provide a complete answer in Section~\ref{descrfinal}: the spectrum of the dynamical system of each (non-trivial) ergodic invariant
measure contains the group of $b_1'\cdot\ldots\cdot b'_k$-roots of unity, $k\geq1$, where in addition to~\eqref{in1} we have
\beq\label{in2}
\sum_{k\geq1}1/b'_k<+\infty.
\eeq
Moreover, the Mirsky measure $\nu_{\mathscr{B}'}$ determined by $\mathscr{B}'=\{b'_k :k\geq 1\}$ yields exactly such spectrum. As a corollary, we obtain that, for example, the square-free subshift has no ergodic invariant measure for which the spectrum of the associated dynamical systems consists of all $p_1\cdot\ldots\cdot p_k$-roots of unity, $k\geq1$.

In general, the set $\mathcal{P}^e(S,X_\eta)$ is quite rich. Te see more members of $\cp^e(S,X_\eta)$ other than simply Mirsky measures of free subsystems of $(S,X_\eta)$, we consider \emph{joining type measures} obtained in the following way. Let $M\colon X_\eta\times\{0,1\}^{\Z}\to X_\eta$ be given by $M(x,u)(n):=x(n)\cdot u(n)$ for $n\in\Z$ (the values of $M$ are in $X_\eta$ because of~\eqref{her}). Let $\lambda$ be an ergodic joining of the Mirsky measure $\nu_{\mathscr{B}'}$ of a \mbox{${\mathscr{B}'}$-free} subshift contained in $(S,X_\eta)$ and an invariant measure $\kappa$ for the full shift $(S,\{0,1\}^{\Z})$\footnote{This means that $\lambda$ is an $S\times S$-invariant ergodic measure on $X_\eta\times\{0,1\}^{\Z}$ such that $\lambda|_{X_\eta}=\nu_{\mathscr{B}'}$ and $\lambda|_{\{0,1\}^{\Z}}=\kappa$.}.
Then the image $M_\ast(\la)$
of $\lambda$ via $M$ belongs to $\cp^e(S,X_\eta)$. Such a measure is called a joining type measure\footnote{\label{stopkac} When $\lambda=\nu_{\mathscr{B}'}\otimes \kappa$ then $M_\ast(\la)$ is called to be of product type; it will be denoted $\nu_{\mathscr{B}'}\ast\kappa$ as it can be  viewed as convolution of measures defined on monoids with the natural coordinate {\bf multiplication}.}.
One of the main results of the paper states that each member of $\cp^e(S,X_\eta)$ is a joining type measure:
\begin{Th}\label{twYall}
For any $\nu\in\mathcal{P}^e(S,X_\eta)$ there  exist a $\mathscr{B}'$-free system and  $\widetilde{\rho}\in\mathcal{P}^e(S\times S,X_\eta\times \{0,1\}^\Z)$ such that $X_{\eta'}\subset X_\eta$, $\widetilde{\rho}|_{X_\eta}=\nu_{\mathscr{B}'}$ and $M_\ast(\widetilde{\rho})=\nu$.
\end{Th}

We also take a closer look at the dynamical systems given by ergodic invariant measures.
We  prove that the measure with maximal entropy yields a system which, up to isomorphism, is the Cartesian product of the discrete spectrum automorphism given by $\nu_\mathscr{B}$ and the Bernoulli system with the entropy $\log 2\cdot\Pi_{i\geq1}(1-1/b_i)$. Moreover, we show that whenever $\kappa\in \mathcal{P}(S,\{0,1\}^\Z)$ yields a system  doubly disjoint\footnote{This forces $\kappa$ to have zero entropy.} \cite{Fu-Pe-We} from the system given by $\nu_{\mathscr{B}}$, then the  map $M$  is an isomorphism of corresponding measure-theoretic dynamical systems, i.e.\ given by $\nu_{\mathscr{B}}\otimes \kappa$ and $\nu_{\mathscr{B}}\ast \kappa$.

Finally, in the last section of the paper, we answer negatively a question raised in \cite{Kw}
whether each hereditary system is intrinsically ergodic.

\section{Intrinsic ergodicity of $\mathscr{B}$-free systems}\label{section1}
\subsection{Basic properties}\label{aaa}

We will recall here some known facts about the dynamical systems associated to $\mathscr{B}$-free numbers. Set $\Omega:=\Pi_{i\geq1} \Z/b_i\Z$. With the product topology and the coordinatewise addition $\Omega$ becomes a compact metrizable Abelian group.
Let $\PP$ stand for the (normalized) Haar measure of $\Omega$ (which is the product of uniform measures on $\Z/b_i\Z$). Denote by $T\colon\Omega\to\Omega$ the homeomorphism given by
$$
T\omega=\omega+(1,1,\ldots)=(\omega(1)+1,\omega(2)+1,\ldots),
$$
where $\omega=(\omega(1),\omega(2),\ldots)$. The dynamical system $(T,\Omega)$ is uniquely ergodic and the measure-theoretic system $(T,\Omega,\cb(\Omega),\PP)$ has discrete spectrum (with the group of eigenvalues equal to the $b_1\cdot\ldots\cdot b_k$-roots of unity, $k\geq1$). In particular,
\begin{equation}\label{eq:zero}
(T,\Omega,\cb(\Omega),\PP)\text{ has zero entropy}.
\end{equation}
Let $\va\colon\Omega\to\{0,1\}^{\Z}$ be defined as
\beq\label{e1}
\va(\omega)(n)=\begin{cases}
1 & \text{if }(\forall i\geq 1)\ \ \  \omega(i)+n\neq 0 \bmod b_i,\\
0 & \text{otherwise}
\end{cases}
\eeq
and let  $\eta:=\va(0,0,\ldots)$. It is easy to check that $\eta$ corresponds  to the characteristic function of the set
$\{m\in\Z : (\forall i\geq1)\ \  \ b_i\nmid m\}$
of $\mathscr{B}$-free numbers.

Following \cite{Sa}, call a subset $A\subset \Z$ {\em admissible} (more precisely, ${\mathscr{B}}$-{\em admissible})  if  $|A  \bmod b_i|<b_i$ for each $i\geq1$. A point $y\in\{0,1\}^{\Z}$ is said to be {\em admissible} if its support
$$
\text{supp}(y):=\{m\in\Z : y(m)=1\}
$$
is admissible.

\begin{Lemma}[\cite{Ab-Le-Ru}, \cite{Sa}] \label{l1}
We have $X_\eta=X_{\mathscr{B}}=\{y\in\{0,1\}^{\Z} : y\;\mbox{is admissible}\}$. Moreover, $\va(\Omega)\subset X_\eta$.
\end{Lemma}
It follows immediately that the subshift $X_\eta$ is hereditary\footnote{This observation was communicated to us by T.\ Downarowicz.} (see \cite{Kw} for basic properties of hereditary systems). This means that if $x\in X_\eta$ and $y\in\{0,1\}^{\Z}$ with $\text{supp}(y)\subset \text{supp}(x)$ then $y\in X_\eta$. Whenever $\text{supp}(y)\subset \text{supp}(x)$, we will write $y\leq x$.
\begin{Lemma}[\cite{Ab-Le-Ru}, cf.\ \cite{Pec}, cf.\ \cite{Sa}]\label{l2}
We have $h_{top}(S,X_\eta)=\log2\cdot \Pi_{i=1}^\infty\left(1-\frac1{b_i}\right)$.
\end{Lemma}
Observe that the map $\va$ is equivariant, i.e.\ $\varphi\circ T=S\circ \varphi$. Moreover, $\varphi$ is Borel but not continuous. Let $\nu_{\mathscr{B}}:=\va_\ast(\PP)$ be the image of $\PP$ via $\va$. Then $\nu_{\mathscr{B}}$ is $S$-invariant. It is called the Mirsky measure of $(S,X_\eta)$, cf.\ \cite{Mirsky}. Let $A\subset\Z$ be non-empty and finite, and set
\beq\label{eq:rozb}
C^j_A:=\{x\in X_\eta : (\forall n\in A)\;\; x(n)=j\},\ j=0,1.
\eeq
As shown in \cite{Ab-Le-Ru},
$$
\nu_{\mathscr{B}}(C^1_A)=\Pi_{i\geq1}\left(1-\frac{|A \bmod b_i|}{b_i}\right).
$$
Using Lemma~\ref{l2}, it follows that
\beq\label{e2}
\nu_{\mathscr{B}}(C^1_{0})= \Pi_{i\geq1}\left(1-\frac{1}{b_i}\right)=h_{top}(S,X_\eta)/\log2.\eeq
Set
\beq\label{ee2}
Y:=\{x\in X_\eta : (\forall i\geq1)\;\;|\text{supp}(x) \bmod b_i|=b_i-1\}.
\eeq
Notice that $Y$ is a Borel set and $SY=Y$. Finally, we have the following:
\begin{Lemma}[cf.\ \cite{Pec}]\label{l3}
Any measure $\nu$ with maximal entropy is concentrated on~$Y$.\footnote{In~\cite{Pec}, the proof is given for $\mathscr{B}=\{p_i^2:i\geq 1\}$. In the general case, the proof goes along the same lines.}
\end{Lemma}

\subsection{A few observations}
Our aim is to show that we can define an ``inverse'' of $\va$ on $Y$. A difficulty is that the image $\va(\Omega)$
of the map $\va\colon\Omega\to X_\eta$ is not ``quite'' included in $Y$ and the map itself is not 1-1. Indeed,
for example the all~0 sequence which does not belong to $Y$ can be
arranged to come about by assigning to each $n \in \Z$ some index $k_n$ in a 1-1 manner and then
choosing $\omega({k_n})\in\Z/b_{k_n}\Z$ so that $\omega({k_n}) + n = 0$ mod $b_{k_n}$ (hence, the fiber $\va^{-1}((\ldots,0,0,\ldots))$ is uncountable). We now show how to bypass this difficulty.

Following \cite{Ab-Le-Ru}, given $k\geq1$ and $z\in\Z/b_k\Z$, we set
\begin{align*}
&\Omega_{k,z}:=\{\omega\in\Omega:\omega(k)=z\},\\
&E_{k,z}:=\{\omega\in\Omega:(\forall s\geq1)\;\va(\omega)(-z+sb_k)=0\}.
\end{align*}
Then $\Omega_{k,z}\subset E_{k,z}$ and
\beq\label{euz1}
T\Omega_{k,z}=\Omega_{k,z+1},\;TE_{k,z}\subset E_{k,z+1}.
\eeq
Moreover,
\beq\label{euz2}
\omega\notin E_{k,z} \mbox{ if and only if }-z+sb_k\in{\rm supp }(\va(\omega))\text{ for some }s\geq 1.
\eeq
Let
$$
\Omega'_0:=\bigcap_{k\geq1}\bigcap_{z\in\Z/b_k\Z}\big(E_{k,z}^c\cup \Omega_{k,z}\big) \text{ and }\Omega_0:=\bigcap_{k\in \Z}T^k\Omega'_0.
$$
Clearly, $\Omega_0$ is a Borel $T$-invariant subset of $\Omega$. We have the following result.

\begin{Lemma}[Proposition 3.2 in \cite{Ab-Le-Ru}]\label{luz1}
We have $\PP(\Omega_0)=1$ and $\va|_{\Omega_0}$ is 1-1.
\end{Lemma}

Define a Borel map $\theta\colon Y\to\Omega$ (cf.\ \cite{Pec}) by setting
\beq\label{e4}
\theta(y)=\omega \text{ if }-\omega(i)\notin \text{supp}(y) \bmod b_i \text{ for all }i\geq 1.
\eeq

\begin{Lemma}\label{l4}
We have:
\begin{enumerate}[(i)]
\item\label{F1}
$\theta$ is equivariant, i.e.\ $T\circ \theta=\theta \circ S$.
\item\label{F2}
For each $y\in Y$ we have $y\leq \va(\theta(y))$.
\item\label{F3}
$\va(\Omega_0)\subset Y$ (in particular, $\theta\circ \va|_{\Omega_0}=id_{\Omega_0}$).
\end{enumerate}
\end{Lemma}
\begin{proof}
\eqref{F1} Note that ${\rm supp}(Sx)={\rm supp}(x) -1$. Hence $-\omega(i)\notin\text{supp}(x) \bmod b_i$ if and only if $-(\omega(i)+1)\notin{\rm supp}(Sx) \bmod b_i$.

\eqref{F2} Suppose that for some $n\in\Z$ we have $y(n)=1$. Then, by~\eqref{e4}, $\theta(y)(k)\neq -n \bmod b_k$ for all $k\geq 1$. In other words, $\theta(y)(k)+n\neq 0 \bmod b_k$ for all $k\geq 1$, i.e.\ $\varphi(\theta(y))(n)=1$.

\eqref{F3} Fix $k\geq1$ and $\omega\in\Omega_0$. We need to prove that $|\text{supp}(\va(\omega))\; {\rm mod}\;b_k|=b_k-1.$
We have $\omega\in \bigcap_{z\in\Z/b_k\Z}\big(E_{k,z}^c\cup \Omega_{k,z}\big)$. Clearly, $\omega\in\Omega_{k,\omega(k)}$ (in particular, $-\omega(k)\notin\text{supp}(\va(\omega))$). Moreover, since $E^c_{k,z}\cap\Omega_{k,z}=\emptyset$, we also have
\begin{equation}\label{GC}
\omega\in \bigcap_{z\in\Z/b_k\Z\setminus\{-\omega(k)\}}E_{k,z}^c.
\end{equation}
It follows from~\eqref{euz2} that, given $z\in\Z/b_k\Z\setminus\{-\omega(k)\}$, for some $s\geq1$, $-z+sb_k\in\text{supp}(\va(\omega))$, whence $-z\in\text{supp}(\va(\omega))$ mod $b_k$ which completes the proof.
\end{proof}

\begin{Remark}\label{ruz1} When $\omega\in\Omega_0$ then, of course, $T^{b_k}\omega\in\Omega_0$. We also have $(T^{b_k}\omega)(k)=\omega(k)$. It follows by~\eqref{GC} that
$$
\omega, T^{b_k}\omega\in  \bigcap_{z\in\Z/b_k\Z\setminus\{-\omega(k)\}}E^c_{k,z}.$$
In view of \eqref{euz2} (applied to $T^{b_k}\omega$), for each $z\neq-\omega(k)$ there exists $s\geq1$ such that
$$
\va(\omega)(-z+(s+1)b_k)=S^{b_k}(\va(\omega))(-z+sb_k)= \va(T^{b_k}\omega)(-z+sb_k)=1.$$
By considering $T^{mb_k}\omega$, $m\geq1$, we conclude: for each $z\in\Z/b_k\Z\setminus\{-\omega(k)\}$
\beq\label{niesk}
\omega\in\Omega_0\;\Rightarrow\;\big|\{s\geq 1:-z+sb_k\in{\rm supp}\, \va(\omega)\}\big|=\infty.\eeq
It follows that
\beq\label{euz11}
\omega\in\Omega_0\;\Rightarrow\; (\text{supp}(\va(\omega))\setminus E)\;{\rm mod}\; b_k=\Z/b_k\Z\setminus\{-\omega(k)\}\eeq
for each finite set $E\subset\Z$ and each $k\geq1$. In particular, for $\omega\in\Omega_0$,
\begin{equation}\label{euz12}
\text{if  $y\leq \varphi(\omega)$ and $\big|\{r\in\Z:y(r)\neq\va(\omega)(r)\}\big|<\infty$ then $y\in Y$.}
\end{equation}
Finally, notice that
\beq\label{euz14}
\text{if }y\in Y\text{ and }y\leq\va(\omega)\text{ then } \theta(y)=\omega.
\eeq
\end{Remark}

\begin{Remark}\label{embedding}
We can repeat the proof of~\eqref{F2} in Lemma~\ref{l4} to obtain the following:
\beq\label{emb1}
\mbox{for each $x\in X_\eta$ there is $\omega\in\Omega$ such that $x\leq\varphi(\omega)$.}\eeq
Indeed, since $x\in X_\eta$ is admissible, for each $k\geq1$, choose
$$
a_k\in(\Z/b_k\Z)\setminus(\text{supp}(x)\;{\rm mod}\;b_k)
$$
and set $\omega:=(-a_1,-a_2,\ldots)$. Then $x\leq\varphi(\omega)$. It follows that
\beq\label{emb2}
X_\eta=\{x\in\{0,1\}^{\Z}:(\exists \omega\in\Omega)\;\;x\leq \varphi(\omega)\}.\eeq
Less formally, we can phrase this by saying that $X_\eta$ is the hereditary system generated by $\va(\Omega)$, that is, generated by the symbolic ``model'' $(S,\va(\Omega))$ of the odometer  $(T,\Omega)$.
\begin{Remark}
The odometer $(T,\Omega)$ is entirely determined by its group of eigenvalues: the group of $b_1\cdot\ldots\cdot b_k$-roots of unity, $k\geq1$. Note that we can replace $\{b_k:k\geq 1\}$ with $\mathscr{B}'=\{b'_j:j\geq1\}$ so that the corresponding odometers $(T,\Omega)$ and $(T',\Omega')$ are topologically conjugate and~\eqref{f1} holds for $\mathscr{B}'$. For example, $b'_1=b_1\cdot\ldots b_{i_1}$, $b'_2=b_{i_1+1}\cdot\ldots\cdot b_{i_2}$, $\ldots$ In this way we obtain a new hereditary system $(S,X_{\eta'})$, cf.\ \eqref{emb2}, which in general will not be conjugated to $(S,X_\eta)$ because $h_{top}(S,X_{\eta'})$ will be different from $h_{top}(S,X_\eta)$, see Lemma~\ref{l2}. In this way, we can obtain a hereditary system $(S,X_{\eta'})$ generated by a symbolic model $(S,\varphi'(\Omega'))$ of the odometer $(T,\Omega)$ which has the entropy arbitrarily close to $\log2$.
\end{Remark}
Notice also that
\beq\label{mz0}\begin{array}{l}
\mbox{whenever $\omega\neq \omega^{\prime\prime}$ are two points from $\Omega_0$ then}\\
\mbox{$\va(\omega)$ and $\va(\omega^{\prime\prime})$  are not $\leq$-comparable.}\end{array}\eeq
Indeed, there exists $k\geq1$ such that $\omega(k)\neq\omega^{\prime\prime}(k)$. Moreover,
$$
\omega,\omega^{\prime\prime}\in \bigcap_{z\in \Z/b_k\Z} (E_{k,z}^c\cup\Omega_{k,z}).
$$
Hence,  $\omega\in \Omega_{k,\omega(k)}$, $\omega^{\prime\prime}\in E^c_{k,\omega(k)}$. It follows that
$\va(\omega)(-\omega(k)+sb_k)=0$ for each $s\geq1$, while there exists $s_0\geq1$ such that $\va(\omega^{\prime\prime})(-\omega(k)+s_0b_k)=1$, see~\eqref{euz2}. Hence, if $\va(\omega),\va(\omega^{\prime\prime})$ are $\leq$-comparable, it must be $\va(\omega)\leq \va(\omega^{\prime\prime})$, and, by symmetry, we obtain equality. Finally, $\omega=\omega^{\prime\prime}$ since $\va$ is 1-1 on $\Omega_0$.
\end{Remark}
Fix a measure $\nu$ on $Y$ with maximal entropy: $h_{\nu}(S,X_\eta)=\log2\cdot\Pi_{i\geq1}\left(1-\frac1{b_i}\right)$ (cf.\ Lemma~\ref{l3}).
\begin{Lemma}\label{l5}
We have
$\theta_\ast(\nu)=\PP$.
\end{Lemma}
\begin{proof}
This follows directly from Lemma~\ref{l4}~\eqref{F1} and the fact that $(T,\Omega)$ is uniquely ergodic.
\end{proof}

Let $Y_0:=\theta^{-1}(\Omega_0)$. Then, by Lemma~\ref{l4}~\eqref{F1}, $Y_0$ is an $S$-invariant Borel subset of $Y$. Now, by Lemma~\ref{luz1} and Lemma~\ref{l5}, we have
\begin{equation}\label{igrekzero}
\nu(Y_0)=\theta_\ast(\nu)(\Omega_0)=\mathbb{P}(\Omega_0)=1.
\end{equation}

Let $Q=(Q_0,Q_1)$ be the partition of $Y$  according to the value at the zero coordinate, i.e.\ $Q_j=C^j_0\cap Y$, $j=0,1$. This is a generating partition. Set
$$
Q^-:=\bigvee_{-\infty}^{-1}S^jQ\text{ and }\ca:=\theta^{-1}(\cb(\Omega)).
$$

\begin{Remark}\label{uw:1}
Since $Q$ is a generating partition, the $\sigma$-algebra $\bigcap_{m\geq0}S^{-m}Q^-$ is the Pinsker $\sigma$-algebra of $(S,Y,\cb(Y),\nu)$ (see e.g.\ \cite{Gl}, Thm.\ 18.9).
\end{Remark}

\begin{Lemma}\label{l6}
We have $\ca\subset \bigcap_{m\geq0}S^{-m}Q^-$ modulo $\nu$.
\end{Lemma}
\begin{proof}
In view of Remark~\ref{uw:1}, the result follows from~\eqref{eq:zero} and from Lemma~\ref{l5}.
\end{proof}

It follows that a.e.\ atom of the partition corresponding to the Pinsker \mbox{$\sigma$-algebra} of $(S,Y,\cb(Y),\nu)$ is contained in an atom of the partition of $Y$ corresponding to $\ca$. We also have $$
\ca\subset S^{-m}Q^-\text{ for }m\geq0,1
$$
so, in other words, after removing a set of $\nu$-measure zero from $Y$, for the remaining points in $Y$ we have the following: for each $m\geq1$
\beq\label{e10}
\left[y_1,y_2\in Y,\;(\forall j\leq-m)\;\; y_1(j)=y_2(j)\right]\;\Rightarrow\; \theta(y_1)=\theta(y_2).
\eeq

Fix $m\geq0$. Let $\pi_m$ be the natural quotient map from $Y$ to $Y/S^{-m}Q^-$.
Let $\ov\nu_m:=(\pi_m)_\ast(\nu)$. Notice that
$S^{-1}(S^{-m})Q^-\subset S^{-m}Q^-$,
so $S$ acts naturally on the quotient space $Y/S^{-m}Q^-$ as an endomorphism preserving $\ov\nu_m$. Moreover, the map $\pi_m$ is equivariant, i.e.\
$$
\pi_m\circ S=S\circ \pi_m.
$$
Using~\eqref{e10}, we may also define the quotient map $\rho_m\colon Y/S^{-m}Q^-\to\Omega$ which is equivariant as well. Then  $(\rho_m)_\ast(\ov\nu_m)=\PP$. In other words, we have the following commuting diagram (in which $\theta,\pi_m$ and $\rho_m$ are measure-preserving while $\va\colon\Omega\to Y$ is defined $\PP$-a.e.\ and is not measure-preserving):

\begin{center}
\begin{tikzpicture}
	\node (U)  at ( 0, 3.5) {$(Y,\nu)$};
	\node (M) at (0,2) {$(Y/{S^{-m}Q^-},\ov\nu_m)$};
	\node (B) at (0,0.5) {$(\Omega,\PP)$};
	\draw[->] (U) edge node[auto] {$\pi_m$} (M)
		       (M) edge node[auto] {$\rho_m$} (B)
		       (U) edge[loop above] node[above] {$S$} (U)
		       (M) edge[loop right] node[left] {$S$} (M)
		       (B) edge[loop below] node {$T$} (B);
	\draw[<-] (U) to ++(3,0)   to node[right, midway]{$\varphi$}++(0,-3)  to(B) ;	
	\draw[->] (U) to ++(-2,0) to node[left,midway] {$\theta$} ++ (0,-3) to (B);	
\end{tikzpicture}
\end{center}

\begin{Remark}
Since the maps $\pi_m\colon Y\to Y/S^{-m}Q^-$, $\rho_m\colon Y/S^{-m}Q^-\to\Omega$ and $\va\colon \Omega\to Y$ are equivariant, it follows immediately that $S^k\circ \va\circ \rho_m=\va\circ\rho_m\circ S^k$. Therefore, if $\ov z\in Y/S^{-m}Q^-$ then
\beq\label{ue1}
\va\circ \rho_m(\ov z)(m+k)=\va(\rho_m(S^k\ov z))(m)\;\;\mbox{for every}\;\;k\in\Z.
\eeq
\end{Remark}
We will identify points in $Y/S^{-m}Q^-$ with their $Q(-\infty,-m-1]$-names: for $y\in Y$, let $\ov{y}$ be the atom of the partition associated to $S^{-m}Q^-$ which contains $y$, i.e.\
$$
\ov{y}=\ldots i_{-1}i_0 \iff y\in S^{-m-1}Q_{i_0}\cap S^{-m-2}Q_{i_{-1}}\cap\ldots.
$$

The following observation is well-known.

\begin{Lemma}\label{l8} For each $m\geq 0$, for each $r=0,1,\ldots,2m$ and $\ov \nu_m$-a.e.\ $\ov y\in Y/S^{-m}Q^-$, we have
\begin{multline}\label{ue2}
 \ov\nu_m\left( S^{m-r}Q_{i_{m-r}}|\right.\\
\left.  S^{m-r-1}Q_{i_{m-r-1}} \cap \ldots \cap S^{-m+1}Q_{i_{-m+1}}\cap S^{-m}Q_{i_m}\cap S^{-m}Q^{-}\right)(\ov{y})\\
=\ov\nu_m\left( S^{-m}Q_{i_{m-r}}| S^{-m}Q^-\right)(\ov{y}i_{-m}\ldots i_{m-r-1})
\end{multline}
for each choice of $i_k\in\{0,1\}$, $-m\leq k\leq m$.
\end{Lemma}
\begin{proof}
Assume that $\ov y=\ldots j_{-m-2}j_{-m-1}$. For $\ov \nu$-a.e.\ such a $\ov y$, by stationarity, we have
\begin{align*}
\ov\nu_m&\left( S^{m-r}Q_{i_{m-r}}|  S^{m-r-1}Q_{i_{m-r-1}} \cap \ldots \cap S^{-m+1}Q_{i_{-m+1}}\cap S^{-m}Q_{i_m}\right.\\
&\left.\cap S^{-m}Q^{-}\right)(\ov{y})\\
=&\lim_{t\to \infty} \nu\left( S^{m-r}Q_{i_{m-r}}|  S^{m-r-1}Q_{i_{m-r-1}} \cap \ldots \cap S^{-m+1}Q_{i_{-m+1}}\cap S^{-m}Q_{i_m}\right.\\
&\left.\cap S^{-m}\left( S^{-1}Q_{j_{-m-1}}\cap\ldots\cap S^{-t}Q_{j_{-m-t}}\right)\right)\\
=&\lim_{t\to \infty} \nu\left( S^{-m}Q_{i_{m-r}}|  S^{-m}(S^{-1}Q_{i_{m-r-1}} \cap \ldots \cap S^{-2m+r}Q_{i_m}\right.\\
&\left.\cap  S^{-2m+r-1}Q_{j_{-m-1}}\cap\ldots\cap S^{-2m+r-t}Q_{j_{-m-t}})\right)\\
=&\ov\nu_m\left(S^{-m}Q_{i_{m-r}}|S^{-m}Q^-\right)(\ldots j_{-m-2}j_{-m-1}i_{-m}\ldots i_{m-r-1})\\
=&\ov\nu_m\left(S^{-m}Q_{i_{m-r}}|S^{-m}Q^-\right)
(\ov{y}i_{-m}\ldots i_{m-r-1})
\end{align*}
which completes the proof.
\end{proof}

\subsection{Proof of Theorem~\ref{benji}}
\subsubsection{Outline of the proof}\label{outline}
Let $\nu$ be a measure of maximal entropy for $(S,X_\eta)$. In order to prove Theorem~\ref{benji}, we will show that the conditional measures $\nu_\omega$ in the disintegration (cf.\ Lemma~\ref{l5})
\begin{equation}\label{miaranu}
\nu=\int_{\Omega}\nu_\omega\ d\PP(\omega)
\end{equation}
of $\nu$ over $\PP$ given by the mapping $\theta\colon Y\to\Omega$ are unique $\PP$-a.e. This will yield intrinsic ergodicity for $(S,X_\eta)$. In fact, we will show that $\nu=\mu$, where the measure $\mu$ is defined in the following way. Recall first that for $\omega\in\Omega_0$, we have $\varphi(\omega)\in Y$. Moreover, in view of Lemma~\ref{l4}~\eqref{F2}, $\varphi(\omega)$ is the largest element in $\theta^{-1}(\omega)$.  In particular, for each $y\in\theta^{-1}(\omega)$, $y[-k,k]\leq \va(\omega)[-k,k]$. Therefore, there are at most $2^m$ blocks $u=(u_{-k},\dots,u_k)$ on $\theta^{-1}(\omega)$, $m=\sum_{j=-k}^{k}\varphi(\omega)(j)$, obtained by replacing some of the $1$s in $\varphi(\omega)[-k,k]$ by $0$s. In fact, in view of~\eqref{euz12} and~\eqref{euz14} all such blocks do occur on $\theta^{-1}(\omega)$. For $u=(u_{-k},\ldots,u_k)\in\{0,1\}^{2k+1}$ denote by $[u]$ the corresponding cylinder set. If $u$ is such that $u\leq \varphi(\omega)[-k,k]$, we set $\mu_\omega([u]):=2^{-m}$, where $m$ has been defined above. Thus, the measure $\mu_\omega$ is equidistributed on all $(2k+1)$-blocks which occur on $\theta^{-1}(\omega)$ for $\omega\in\Omega_0$. Finally, we set
$$
\mu=\int_{\Omega}\mu_\omega\ d\PP(\omega).
$$
In a less formal way, a random point distributed according to $\mu$ is obtained by first choosing an $\omega\in\Omega$ according to $\PP$ and then for each $n\in\Z$, where $\va(\omega)(n)=1$ changing the $1$ to $0$ with probability $1/2$, independently for all such $n$.

We will show that for $\PP$-a.e.\ $\omega\in \Omega$, $\mu_\omega=\nu_\omega$. In order to do it, we will show that for $A$ belonging to a countable dense family of subsets in $\cb$, we have
\begin{equation}\label{plan1}\mbox{$\nu_\omega(A)=\mu_\omega(A)$ for $\PP$-a.e.\ $\omega\in \Omega$.}\end{equation}

Recall that
\begin{equation}\refstepcounter{equation}\subeqn\label{Ka}
\nu_\omega(A)=\E^\nu(A|\Omega)(\omega).
\end{equation}
To get the equality~\eqref{plan1}, we will step by step make use of the equality
\begin{equation}\subeqn\label{Kb}
\E^\nu(A|\Omega)(\omega)=\E^{\nu}\left(\E^\nu(A|Y/S^{-m}Q^-)
(\overline{y}_m)\big|\Omega\right)(\omega)
\end{equation}
where $A\in \bigvee_{t=-m}^mS^tQ$, $m\geq0$ and  show that
\begin{equation}\subeqn\label{Kc}
\E^\nu(A|Y/S^{-m}Q^-)(\overline{y}_m)=\mu_\omega(A)\\
\end{equation}
for all $\overline{y}_m$  having the same $\rho_m$-projection $\omega$  (for this, we use~\eqref{e2} and a convexity argument on the entropy). The proof will go as follows:
\begin{itemize}
\item
we first show that \eqref{plan1} holds for  $A\in Q$, that is, for $m=0$;
\item
we show that  \eqref{plan1} is satisfied  for $A\in \bigvee_{t=-m}^mS^tQ$  for any $m\geq0$.
\end{itemize}
The first of the above  steps is not necessary -- it can be seen as a toy model for the second step. However, we include it to increase readability. In what follows we identify $Y$ with $Y_0$ and $\Omega$ with $\Omega_0$.

\subsubsection{Toy model: $\nu_\omega=\mu_\omega$ on $Q$}
Let
\begin{equation}\label{e:1}
\widehat{C}^j_0:=\va^{-1}(C^j_0)=\{\omega\in\Omega : \varphi(w)(0)=j\}\text{ for }j=0,1
\end{equation}
(recall that the sets $C_0^j$ were defined in~\eqref{eq:rozb}). Then  $\Omega=\widehat{C}^0_0 \cup \widehat{C}^1_0$.
Moreover,
\begin{equation}\label{eq:rozY}
Y=\theta^{-1}(\Omega)=\theta^{-1}(\widehat{C}^0_0)\cup \theta^{-1}(\widehat{C}^1_0).
\end{equation}
It follows from~\eqref{e:1} that for $j=0,1$, we have
\begin{equation}\label{e:2}
\theta^{-1}(\widehat{C}^j_0)=\bigcup_{\omega\in\Omega,
\va(\omega)(0)=j}\theta^{-1}(\omega).
\end{equation}
In other words,~\eqref{eq:rozY} is the partition of $Y$ given by the fibers $\theta^{-1}(\omega)$ of $\theta$, according to the value at the zero coordinate of the biggest element $\varphi(\omega)$ in the fiber, cf.\ Lemma~\ref{l4}~(ii). Finally, let
$$
B_0^j:=\rho^{-1}_0(\widehat{C}^j_0)\text{ for }j=0,1,
$$
i.e.\ we have
$$
Y/Q^-=B_0^0\cup B_0^1.
$$
This can be summarized in the following diagram ($\va$ is not measure-preserving):
\begin{center}
\begin{tikzpicture}
	\node (U) at (0,4) {$Y$};
	\node (Ur) at (1.8,4.02) {$=\theta^{-1}(\widehat{C}^0_0)\cup \theta^{-1}(\widehat{C}^1_0)$};
	\node (Ul) at (-1,3.98) {$Q_0\cup Q_1=$};
	\node (M) at (0,2) {$Y/{Q^-}$};
	\node (Mr) at (3,2) {$=\rho^{-1}_0(\widehat{C}^0_0)\cup\rho^{-1}_0(\widehat{C}^1_0)=B^0_0\cup B^1_0$};
	\node (B) at (0,0) {$\Omega$};
	\node (Br) at (1,0) {$=\widehat{C}^0_0 \cup \widehat{C}^1_0$};
	\draw[->] (U) edge node[auto] {$\pi_0$} (M)
		       (M) edge node[auto] {$\rho_0$} (B);
	\draw[<-] (Ur) to ++(5,0)   to node[right, midway]{$\varphi$}++(0,-4)  to(Br) ;	
	\draw[->] (Ul) to ++(-3,0) to node[left,midway] {$\theta$} ++ (0,-4) to (B);
\end{tikzpicture}
\end{center}

We claim that $\theta^{-1}(\widehat{C}^0_0)\subset Q_0$. Indeed, in view of~\eqref{e:2}, if $y\in \theta^{-1}(\widehat{C}^0_0)$ then $y\in \theta^{-1}(\omega)$, where $\va(\omega)(0)=0$, i.e.\ $\va(\omega)\in Q_0.$ Since $y\leq \va(\omega)$, $y\in Q_0$. It follows that for $\ov y\in B^0_0=\rho_0^{-1}(\hat{C}_0^0)$ we have
$$
\pi_0^{-1}(\ov y)\subset \pi_0^{-1}\rho_0^{-1}(\hat{C}_0^0)=\theta^{-1}(\hat{C}_0^0)\subset Q_0
$$
and therefore
$$
(Q_0,Q_1)\cap \pi_0^{-1}(\ov y)=(\pi_0^{-1}(\ov y),\emptyset).
$$
Since the measure $\ov{\nu}_0(\cdot|Q^-)(\ov y)$ is concentrated on $\pi_0^{-1}(\ov y)$,
we  obtain for each $\ov y\in B^0_0$
\beq\label{ue33}
\big(\ov{\nu}_0(Q_0|Q^-)(\ov y), \ov{\nu}_0(Q_1|Q^-)(\ov y)\big)=(1,0)=:(\la_0(Q_0),\la_0(Q_1)).
\eeq
Therefore, $H_\nu(Q|Q^-)(\ov y)=0$ whenever $\ov y\in B^0_0$.

Now, in view of Lemma~\ref{l5} and the definition of $\nu_{\mathscr{B}}$, we obtain
$$
\nu_{\mathscr{B}}(C^1_0)=\PP(\widehat{C}^1_0)=
\nu(\theta^{-1}(\widehat{C}^1_0)).$$
Hence, using additionally~\eqref{e2}, we have
\begin{align*}
\nu(\theta^{-1}(\widehat{C}^1_0))\log2&=h_{top}(S,X_\eta)=h_\nu(S,X_\eta)\\
&=\int_{Y/Q^-} H_\nu(Q|Q^-)(\ov y)\ d\ov{\nu}_0(\ov y)\\
&=\int_{B^0_0}H_\nu(Q|Q^-)(\ov y)\ d\ov{\nu}_0(\ov y)+\int_{B^1_0}H_\nu(Q|Q^-)(\ov y)\ d\ov{\nu}_0(\ov y)\\
&=\int_{B^1_0}H_\nu(Q|Q^-)(\ov y)\ d\ov{\nu}_0(\ov y)\leq \ov\nu_0(B^1_0)\log2= \nu(\theta^{-1}(\widehat{C}^1_0))\log2.
\end{align*}
It follows that for $\ov\nu_0$-a.e.\ $\ov y\in B^1_0$, we have
$ H_\nu(Q|Q^-)(\ov y)=\log2$,
or, equivalently
\beq\label{ue44}
\left(\ov{\nu}_0(Q_0|Q^-)(\ov y), \ov{\nu}_0(Q_1|Q^-)(\ov y)\right)=(1/2,1/2)=:\left(\la_1(Q_0),\la_1(Q_1)\right).
\eeq

Both~\eqref{ue33} and~\eqref{ue44} do not depend on $\ov y$ itself but only on the value $\va(\rho_0(\ov y))(0)$ which allows us to make use of~\eqref{Kc}. We now use~\eqref{Ka},~\eqref{Kb} and~\eqref{Kc} to conclude that in the disintegration~\eqref{miaranu} of $\nu$ over $\PP$ (via $\theta$), for $\PP$-a.e.\ $\omega\in\Omega$, $\nu_\omega(Q_i)=\mu_\omega(Q_i)$ for $i=0,1$ (in view of ~\eqref{ue33} and~\eqref{ue44}).

\subsubsection{General case: $\nu_\omega=\mu_\omega$ on $\bigvee_{t=-m}^mS^tQ$}
Now, fix $m\geq 0$ and let
$$
\widehat{C}^j_{m}:=\va^{-1}(C^j_{m})=\{\omega\in\Omega : \varphi(w)(m)=j\}\text{ for }j=0,1.
$$
Then $\Omega=\widehat{C}^0_{m}\cup \widehat{C}^1_{m}$ and
$$
Y=\theta^{-1}(\Omega)=\theta^{-1}(\widehat{C}^0_{m})\cup \theta^{-1}(\widehat{C}^1_{m}).
$$
We have
\beq\label{eq:13a}
\theta^{-1}(\widehat{C}^j_{m})=\cup_{\omega\in \Omega, \varphi(\omega)(m)=j}\theta^{-1}(\omega).
\eeq
Finally, let
$$
B^j_{m}:=\rho^{-1}_m(\widehat{C}^j_{m})\text{ for }j=0,1,
$$
i.e.\ we have
$$
Y/S^{-m}Q^-=B_{m}^0\cup B_{m}^1.
$$
This can be summarized in the following diagram:
\begin{center}
\begin{tikzpicture}
	\node (U) at (0,4) {$Y$};
	\node (Ur) at (1.9,4) {$=\theta^{-1}(\widehat{C}^0_{m})\cup \theta^{-1}(\widehat{C}^1_{m})$};
	\node (Ul) at (-1.7,3.98) {$S^{-m}Q_0\cup S^{-m}Q_1=$};
	\node (M) at (0,2) {$Y/{S^{-m}Q^-}$};
	\node (Mr) at (3.5,2) {$=\rho^{-1}_{m}(\widehat{C}^0_{m})\cup\rho^{-1}_m(\widehat{C}^1_{m})=B^0_{m}\cup B^1_{m}$};
	\node (B) at (0,0) {$\Omega$};
	\node (Br) at (1.1,0) {$=\widehat{C}^0_{m} \cup \widehat{C}^1_{m}$};
	\draw[->] (U) edge node[auto] {$\pi_m$} (M)
		       (M) edge node[auto] {$\rho_m$} (B);
	\draw[<-] (Ur) to ++(5.3,0)   to node[right, midway]{$\varphi$}++(0,-4)  to(Br) ;	
	\draw[->] (Ul) to ++(-1.8,0) to node[left,midway] {$\theta$} ++ (0,-4) to (B);
\end{tikzpicture}
\end{center}
As in the toy model case, we obtain that $\theta^{-1}(\widehat{C}^0_{m})\subset S^{-m}Q_0$, whence, for each $\ov y\in B^0_{m}$,
\begin{multline}\label{ue3}
\big(\ov{\nu}_m(S^{-m}Q_0|S^{-m}Q^-)(\ov y), \ov{\nu}_m(S^{-m}Q_1|S^{-m}Q^-)(\ov y)\big)\\
=(1,0)=(\la_0(Q_0),\la_0(Q_1)).
\end{multline}
Therefore, $H_\nu(S^{-m}Q|S^{-m}Q^-)(\ov y)=0$ whenever $\ov y\in B^0_{m}$. Since $S^{-m}Q$ is a generating partition whose past is equal to $S^{-m}Q^-$, the computation of
$$
\int_{Y/S^{-m}Q^-}H_\nu(S^{-m}Q|S^{-m}Q^-)\ d\ov\nu_m
$$
similar to the toy model case leads to
\begin{multline}\label{eq:18}
\left(\ov{\nu}_m(S^{-m}Q_0|S^{-m}Q^-)(\ov y), \ov{\nu}_m(S^{-m}Q_1|S^{-m}Q^-)(\ov y)\right)\\
=(1/2,1/2)=\left(\la_1(Q_0),\la_1(Q_1)\right)
\end{multline}
for $\ov\nu_m$-a.e.\ $\ov y\in B^1_{m}$.

We claim that
\begin{equation}\label{ff}
\nu_\omega=\mu_\omega\text{ on }\bigvee_{t=-m}^mS^tQ.
\end{equation}
In order to prove this, choose $(i_{-m},\ldots,i_0,\ldots,i_{m})\in \{0,1\}^{2m+1}$.
By the chain rule for conditional probabilities and Lemma~\ref{l8}, we obtain
\begin{align*}
\ov\nu_m&(S^mQ_{i_m}\cap\ldots\cap Q_{i_0}\cap S^{-1}Q_{i_{-1}}\cap\ldots\cap S^{-m}Q_{i_{-m}}\big |S^{-m}Q^-)(\ov y)\\
&=\Pi_{r=0}^{2m}\ov\nu_m(S^{m-r}Q_{i_{m-r}}\big|S^{m-r-1}Q_{i_{m-r-1}}
\cap\ldots\cap S^{-m}Q_{i_{-m}}\cap S^{-m}Q^-)(\ov y)\\
&=\Pi_{r=0}^{2m}\ov\nu_m(S^{-m}Q_{i_{m-r}}\big|S^{-m}Q^-)(\ov{y}i_{-m}
\ldots i_{m-r-1}).
\end{align*}
It follows from~\eqref{ue3} and~\eqref{eq:18} that for $\nu_m$-a.e. $\ov y$
$$
\ov\nu_m(S^{-m}Q_{i_{m-r}}\big|S^{-m}Q^-)(\ov{y}i_{-m}\ldots i_{m-r-1})=(\la_{j_r}(Q_0),\la_{j_r}(Q_1)),
$$
where $j_r=\va(\rho_m(\ov{y}i_{-m}\ldots i_{m-r-1}))(m)$ (see the definition of $B_{m}^r$). Using~\eqref{ue1}, we hence obtain $j_r=\va(\rho_m(\ov y))(m+r)$. As in the toy model, using~\eqref{Ka},~\eqref{Kb} and~\eqref{Kc}, we obtain~\eqref{ff}.

Carrying this out for all $m\in\N$, we will show that for $\PP$-a.e.\ $\omega\in\Omega$, $\mu_\omega=\nu_\omega$ and hence $\mu=\nu$ as required. The proof of Theorem~\ref{benji} is complete.

\section{Invariant measures for $\mathscr{B}$-free systems}
\label{section2}

Recall that $M\colon X_\eta\times \{0,1\}^\Z\to X_\eta$ is given by
$$
M(x\cdot u)(n)=x(n)\cdot u(n),\;n\in\Z.
$$
Since $M$ is equivariant, for each $\rho\in\mathcal{P}^e(S\times S,X_\eta\times \{0,1\}^\Z)$, we have $M_\ast(\rho)\in \mathcal{P}^e(S,X_\eta)$.
In particular, in the above construction, we can consider measures $\rho$ whose projection onto the first coordinate is the Mirsky measure $\nu_\mathscr{B}$. In fact, instead of $\nu_{\mathscr{B}}$, we can also use the Mirsky measures $\nu_{\mathscr{B}'}$, where $\mathscr{B}'$ is such that the corresponding free system $X_{\eta'}$ is a subsystem of $X_\eta$, see Examples~\ref{przyk1} and~\ref{przyk2} below. We will call the measures of the form $M_\ast(\rho)$, where $\rho\in\mathcal{P}^e(S\times S,X_\eta\times \{0,1\}^\Z)$ $\rho|_{X_\eta}=\nu_{\mathscr{B}'}$ for some $\mathscr{B}'$-free subshift $X_{\mathscr{B}'}\subset X_{\mathscr{B}}$, to be of \emph{joining type}\footnote{Notice that $\rho$, as a member of $\mathcal{P}^e(S\times S,X_\eta\times \{0,1\}^\Z)$, is an ergodic joining of $\nu_{\mathscr{B}'}\in\mathcal{P}^e(S,X_\eta)$ and $\rho|_{\{0,1\}^\Z}\in\mathcal{P}^e(S,\{0,1\}^\Z)$.} (see also footnote~\ref{stopkac}).

The natural question arises whether
$\mathcal{P}^e(S,X_\eta)$ consists only of measures of joining type. We will give a positive answer to this question in Sections~\ref{joa} and~\ref{descrfinal}.

\subsection{Invariant measures on $Y$}
\subsubsection{Ergodic invariant measures on $Y$ are of joining type}\label{joa}
The main result in this section is the following:
\begin{Th}\label{twY}
For any $\nu\in\mathcal{P}^e(S,Y)$ there exists $\widetilde{\rho}\in\mathcal{P}^e(S\times S,X_\eta\times \{0,1\}^\Z)$ such that $\widetilde{\rho}|_{X_\eta}=\nu_{\mathscr{B}}$ and $M_\ast(\widetilde{\rho})=\nu$.
\end{Th}
\begin{Remark}
Some of the objects occurring in the proof will be very similar to the their ``one-sided versions'' described in~\cite{Pec}.
\end{Remark}
\begin{proof}[Proof of Theorem~\ref{twY}]
Notice first that $\nu\neq\delta_{(\dots,0,0,0,\dots)}$. In particular,
\begin{equation}\label{supp-niesk}
\nu(\{y\in Y: |\text{supp}(y) \cap  (-\infty,0)|=|\text{supp}(y) \cap  (0,\infty)|=\infty\})=1.
\end{equation}
For $x,z\in \{0,1\}^\Z$ with
$|\text{supp}\ z \cap (-\infty,0)|=|\text{supp}\ z \cap (0,\infty)|=\infty$,
let $\widehat{x}_z$ be the sequence obtained by reading consecutive coordinates of $x$ which are in $\text{supp}\ z$, and such that
$$
\widehat{x}_z(0)=x(\min\{k\geq 0 : k\in \text{supp}\, z\}).
$$

Let $\Theta\colon Y\to \Omega\times\{0,1\}^\Z$ be given by
$$
\Theta(y)=(\theta(y),\widehat{y}_{\varphi(\theta(y))}).
$$
We have
\begin{equation}\label{eq:Z1}
(\Theta\circ S)y=(\theta Sy,\widehat{Sy}_{\varphi(\theta(Sy))})=(T\theta y,\widehat{Sy}_{S\varphi(\theta(y))}).
\end{equation}
Notice that for $x,z\in\{0,1\}^\Z$, the value of $\widehat{Sx}_{Sz}$ depends on $z(0)$ in the following way:
\begin{equation}\label{26.5.a}
\widehat{Sx}_{Sz}=\begin{cases}
\widehat{x}_z&\text{if } z(0)=0,\\
S\widehat{x}_z&\text{if } z(0)=1
\end{cases}
\end{equation}
(we illustrate this in Figure~\ref{default}).
\begin{figure}[h]
\includegraphics[scale=.45]{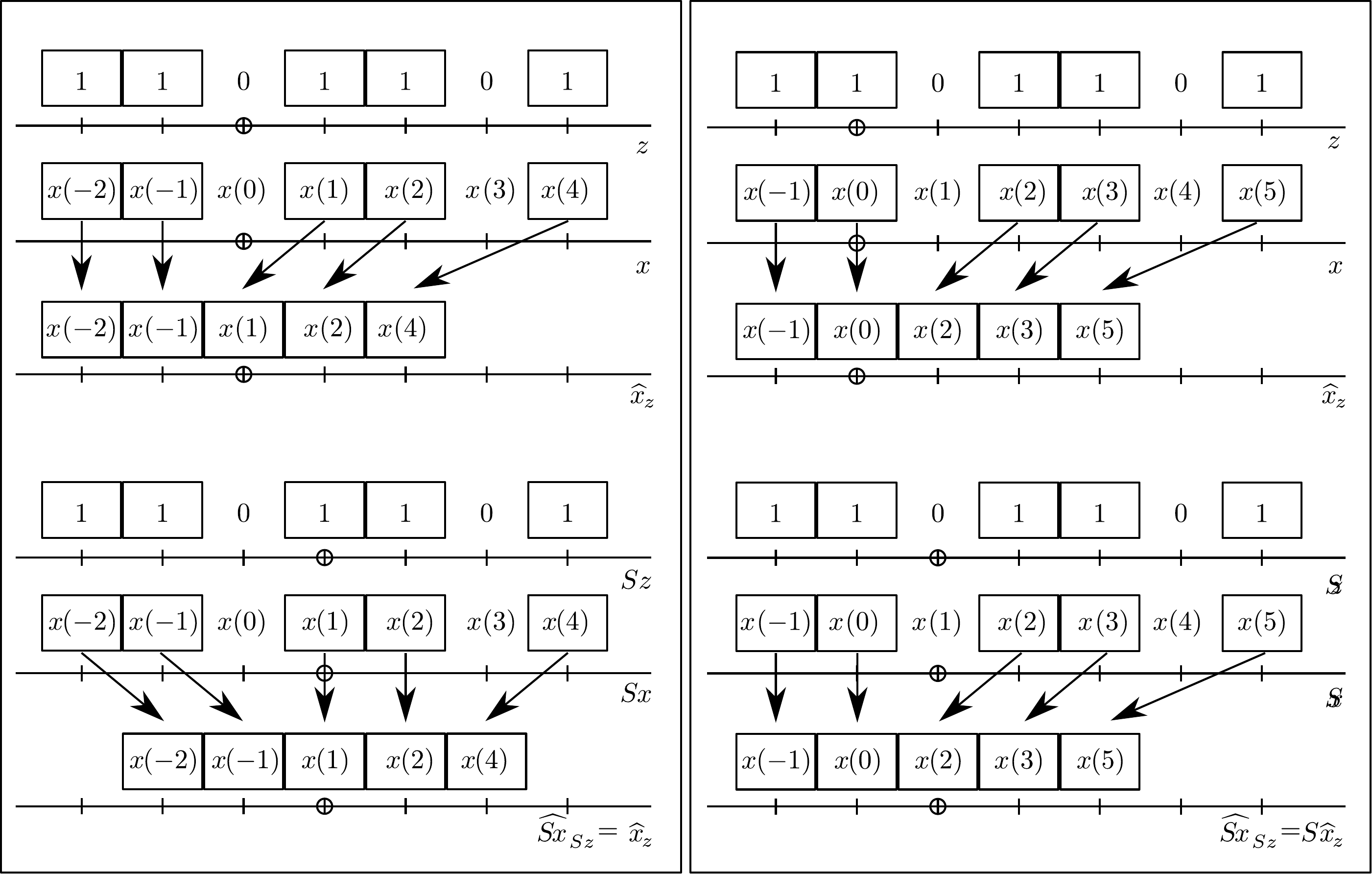}
\caption{Illustration of formula~\eqref{26.5.a}. Case $z(0)=0$ and $z(0)=1$ (the symbol $\oplus$ stands for the zero coordinate).}
\label{default}
\end{figure}
Therefore, it follows from~\eqref{eq:Z1} that
\label{15}
$$
\Theta\circ S|_{Y_\infty} = \widetilde{T}\circ \Theta|_{Y_\infty},
$$
where $\widetilde{T}\colon \Omega\times \{0,1\}^\Z\to \Omega\times \{0,1\}^\Z$ is given by
$$
\widetilde{T}(\omega,x)=\begin{cases}
(T\omega, x)&\text{if }\varphi(\omega)(0)=0,\\
(T\omega, Sx)&\text{if } \varphi(\omega)(0)=1
\end{cases}
$$
and $Y_\infty:=\{y\in Y : |\text{supp\ }\varphi(\theta(y)) \cap (-\infty,0)|= |\text{supp\ }\varphi(\theta(y)) \cap (0,\infty)|=\infty\}$.  Then, since $y\leq \varphi(\theta(y))$, it follows by~\eqref{supp-niesk} that $\nu(Y_\infty)=1$ for any $\nu\in\mathcal{P}^e(S,Y)$. Thus, $\Theta\circ S = \widetilde{T}\circ \Theta$ holds a.e. with respect to any $\nu\in\mathcal{P}^e(S,Y)$:
\begin{equation}\label{diagram1}
\begin{tikzpicture}[baseline=(current  bounding  box.center)]
\node (DL) at (0,0) {$\Omega\times\{0,1\}^{\Z}$};
\node (UL) at (0,1.5) {$Y$};
\node (DR) at (3,0) {$\Omega\times\{0,1\}^{\Z}$.};
\node (UR) at (3,1.5) {$Y$};
\draw[->] (DL) edge node[auto] {$\widetilde{T}$} (DR)
	       (UL) edge node[auto] {$S$} (UR)
	       (UL) edge node[auto] {$\Theta$} (DL)
	       (UR) edge node[auto] {$\Theta$} (DR);
\end{tikzpicture}
\end{equation}

Notice that $\Theta(y)$ contains complete information about each $y\in Y_\infty$:
\begin{itemize}
\item the first coordinate, i.e.\ $\theta(y)$, contains, for each $k$, the information about the missing residue classes in $\text{supp}\ y \bmod b_k$,
\item the second coordinate, i.e.\ $\widehat{y}_{\varphi(\theta(y))}$, contains the information about $y$ along $\text{supp}\ \varphi(\theta(y))$.
\end{itemize}
This allows us to define $\Phi\colon \Theta(\Omega)\to X_\eta$ such that $\Phi\circ\widetilde{T}=S\circ \Phi$.  We do this in the following way: $\Phi(\omega, x)$ is the unique element in $X_\eta$ such that
\begin{enumerate}[(i)]
\item\label{eq:C1}
$\Phi(\omega,x)\leq \varphi(\omega)$,
\item
$\widehat{\Phi(\omega,x)}_{\varphi(\omega)}=x$, i.e.\ the consecutive coordinates of $x$ can be found in $\Phi(\omega,x)$ along $\varphi(\omega)$.
\end{enumerate}
Notice that if follows from~\eqref{niesk}\footnote{In a similar way as in~\eqref{niesk}, we have $\big|\{s\leq -1:-z+sb_k\in{\rm supp}\, \va(\omega)\}\big|=\infty$ for $\omega\in \Omega_0$.} that $\Omega_0\times \{0,1\}^\Z \subset \Theta(Y_\infty)$, i.e.\ $\Phi(\omega, x)$ is well-defined on $\Omega_0\times \{0,1\}^\Z$. We will show now that the following diagram commutes:
\begin{equation*}
\begin{tikzpicture}[baseline=(current  bounding  box.center)]
\node (DL) at (0,0) {$X_\eta$};
\node (UL) at (0,1.5) {$\Omega_0\times\{0,1\}^{\Z}$};
\node (DR) at (3,0) {$X_\eta$.};
\node (UR) at (3,1.5) {$\Omega_0\times \{0,1\}^\Z$};
\draw[->] (DL) edge node[auto] {$S$} (DR)
	       (UL) edge node[auto] {$\widetilde{T}$} (UR)
	       (UL) edge node[auto] {$\Phi$} (DL)
	       (UR) edge node[auto] {$\Phi$} (DR);
\end{tikzpicture}
\end{equation*}
Indeed, in view of the definition of $\Phi$ and $\widetilde{T}$, $\Phi\circ \widetilde{T}(\omega,x)$ is the unique element in $X_\eta$ such that:
\begin{itemize}
\item
$\Phi\circ \widetilde{T} (\omega,x)\leq \varphi(T\omega)$,
\item
$(\Phi\circ \widetilde{T} (\omega,x))^{\hat{}}_{\varphi(T\omega)}=\begin{cases}
x&\text{if } \varphi(T\omega)(0)=0\\
Sx&\text{if } \varphi(T\omega)(0)=1.
\end{cases}$
\end{itemize}
Moreover, by~\eqref{26.5.a}, we have
\begin{itemize}
\item
$S\circ \Phi (\omega,x)\leq S\varphi(\omega)=\varphi(T\omega)$,
\item
$
(S\circ \Phi(\omega,x))^{\hat{}}_{S\varphi(\omega)}=\begin{cases}
(\Phi(\omega,x))^{\hat{}}_{\varphi(\omega)}&\text{if } \varphi(\omega)(0)=0\\
S((\Phi(\omega,x))^{\hat{}}_{\varphi(\omega)})&\text{if } \varphi(\omega)(0)=1,
\end{cases}$\\
i.e.\ $(\Phi\circ \widetilde{T} (\omega,x))^{\hat{}}_{\varphi(T\omega)}=\begin{cases}
x&\text{if } \varphi(T\omega)(0)=0\\
Sx&\text{if } \varphi(T\omega)(0)=1.
\end{cases}
$
\end{itemize}
Thus, we have obtained $S\circ \Phi=\Phi\circ \widetilde{T}$. Since
$$
\Theta^{-1}(\Omega_0\times \{0,1\}^\Z)=\theta^{-1}(\Omega_0)=Y_0,
$$
where, by~\eqref{igrekzero}, $\nu(Y_0)=1$ for any $\nu\in\mathcal{P}^e(S,Y)$, it follows that the composition $\Phi\circ \Theta$ is well-defined a.e.\ with respect to any $\nu\in\mathcal{P}^e(S,Y)$.
We claim that
\begin{equation}\label{iden}
\Phi\circ \Theta = id_{Y_0}.
\end{equation}
Indeed, we have $\Phi\circ \Theta(y)=\Phi(\theta(y),\widehat{y}_{\varphi(\theta(y))})$.
Moreover, $\Phi(\theta(y),\widehat{y}_{\varphi(\theta(y))})$ is the unique element such that
$$
\Phi(\theta(y),\widehat{y}_{\varphi(\theta(y))})\leq \varphi(\theta(y))\text{ and }({\Phi(\theta(y),\widehat{y}_{\varphi(\theta(y))})})^{\hat{}}_{\varphi(\theta(y))}=\widehat{y}_{\varphi(\theta(y))}.
$$
However, since $y\leq \varphi(\theta(y))$, it follows immediately that
$\Phi(\theta(y),\widehat{y}_{\varphi(\theta(y))})=y$,
which yields~\eqref{iden}. Hence, for each $\nu\in \mathcal{P}^e(S,Y)$, we have
$$
\nu=\Phi_\ast(\Theta_\ast \nu),\text{ where }\Theta_\ast\nu\in\mathcal{P}^e(\widetilde{T},\Omega\times \{0,1\}^\Z).
$$

Notice that we have also the following commuting diagram:
\begin{equation}\label{diagram3}
\begin{tikzpicture}[baseline=(current  bounding  box.center)]
\node (DL) at (0,0) {$\Omega_0\times\{0,1\}^{\Z}$};
\node (UL) at (0,1.5) {$\Omega_0\times\{0,1\}^{\Z}$};
\node (DR) at (3,0) {$\Omega_0\times\{0,1\}^{\Z}$};
\node (UR) at (3,1.5) {$\Omega_0\times\{0,1\}^{\Z}$};
\draw[->] (DL) edge node[auto] {$\widetilde{T}$} (DR)
	       (UL) edge node[auto] {$T\times S$} (UR)
	       (UL) edge node[auto] {$\Psi$} (DL)
	       (UR) edge node[auto] {$\Psi$} (DR);
\end{tikzpicture}
\end{equation}
where $\Psi(\omega,x)=(\omega, \widehat{x}_{\varphi(\omega)})$. Indeed, using~\eqref{26.5.a}, we obtain
\begin{multline*}
\Psi\circ(T\times S)(\omega,x)=\Psi(T\omega,Sx)\\
=(T\omega,\widehat{Sx}_{\varphi(T\omega)})=(T\omega,\widehat{Sx}_{S\varphi(\omega)})=\begin{cases}
(T\omega,\widehat{x}_{\varphi(\omega)})&\text{if } \varphi(\omega)(0)=0\\
(T\omega,S\widehat{x}_{\varphi(\omega)})&\text{if }\varphi(\omega)(0)=1,
\end{cases}
\end{multline*}
whereas
$$
\widetilde{T}\circ \Psi(\omega,x)=\widetilde{T}(\omega,\widehat{x}_{\varphi(\omega)})=\begin{cases}
(T\omega,\widehat{x}_{\varphi(\omega)})&\text{if } \varphi(\omega)(0)=0\\
(T\omega,S\widehat{x}_{\varphi(\omega)})&\text{if }\varphi(\omega)(0)=1.
\end{cases}
$$
Notice that $\emptyset\neq\Psi^{-1}(\omega,y)\subset \{\omega\}\times \{0,1\}^\Z$. Moreover, given $(\omega,x)\in \Psi^{-1}(\omega,y)$, all other points in $\Psi^{-1}(\omega,y)$ are obtained by changing in an arbitrary way these coordinates in $x$ which are not in the support of $\varphi(\omega)$. In particular, each fiber $\Psi^{-1}(\omega,y)$ is infinite. For $k_1<\dots <k_s$ and $(i_1,\dots, i_s)\in \{0,1\}^s$ we define the following cylinder set:
\begin{equation}\label{cylinder}
C=C^{i_1,\dots,i_s}_{k_1,\dots,k_s}:=\{x\in \{0,1\}^\Z : x({k_j})=i_j, 1\leq j\leq s\}.
\end{equation}
For each such $C$ and for $A\in \mathcal{B}(\Omega)$ we put
$$
\lambda_{(\omega,y)}(A\times C):=\raz_A(\omega) \cdot 2^{-m}, \text{ where }m=|\{1\leq j\leq s : \varphi(\omega)(k_j)=0\}|,
$$
whenever $\Phi(\omega,y)$ agrees with $C$ along $\varphi(\omega)$, i.e.
\begin{equation*}\label{eq:C2}
\Phi(\omega,y)(k_j)\cdot \varphi(\omega)(k_j)=C(k_j) \cdot \varphi(\omega)(k_j)
\end{equation*}
(otherwise we set $\lambda_{(\omega,y)}(A\times C):=0$). In view of part~\eqref{eq:C1} of the definition of~$\Phi$, this is equivalent to
$$
 \Phi(\omega,y)(k_j)=i_j\text{ whenever }\varphi(\omega)(k_j)=1.
$$
We claim that the following is true:\label{p17}
\begin{enumerate}[(a)]
\item\label{eq:Cz}
the map $F:(\omega,y)\mapsto \lambda_{(\omega,y)}$ is measurable,
\item\label{eq:C}
$(T\times S)_\ast \lambda_{(\omega,y)}=\lambda_{\widetilde{T}(\omega,y)}$.
\end{enumerate}
For~\eqref{eq:Cz}, it suffices to show that sets of the form
$$
V_{A,C,a,\vep}=\{(\omega,y)\in \Omega_0\times\{0,1\}^\Z : |\lambda_{(\omega,y)}(A\times C)-a|<\vep\}.
$$
are measurable for any $A\in\mathcal{B}(\Omega)$, any cylinder $C$ as in~\eqref{cylinder}, any $a\in\R$ and $\vep>0$. Indeed, for $\underline{\lambda}\in\mathcal{P}(\Omega_0\times \{0,1\}^\Z)$ and $a=\underline{\lambda}(A\times  C)$
$$
V_{A,C,a,\vep}=F^{-1}(\{\lambda\in \mathcal{P}(\Omega\times \{0,1\}^\Z) : |\lambda(A\times C)-\underline{\lambda}(A\times C)|<\vep\}).
$$
Notice that each $V_{A,C,a,\vep}$ is an at most countable union of sets of the form
$$
V_{A,C,b}:=\{(\omega,y) : \lambda_{(\omega,y)}(A\times C)=b\},
$$
where $b\in \{0\}\cup \{2^{-m}: m\geq 0\}$. Let
$$
V_C:=\{(\omega,y) : (\Phi(\omega,y)(k_j)-C(k_j))\cdot \varphi(\omega)(k_j) = 0, 1\leq j\leq s\}.
$$
Then
$$
V_{A,C,0}=(A^c \times \{0,1\}^\Z) \cup V_C^c
$$
and for $m\geq 0$,
$$
V_{A,C,2^{-m}}=(A\times  \{0,1\}^\Z)\cap V_C \cap \Big\{(\omega,y): \sum_{j=-k}^{k}\varphi(\omega)(j)=2k+1-m\Big\}.
$$
This implies measurability of the sets $V_{A,C,a}$ as $\varphi$ and $\Phi$ are measurable. To see that also~\eqref{eq:C} holds, notice first that we have
$$
(T\times S)_\ast \lambda_{(\omega,y)}(A\times C)=\lambda_{(\omega,y)}(T^{-1}A\times S^{-1}C)
$$
and
$$
\lambda_{\widetilde{T}(\omega,y)}(A\times C)=\begin{cases}
\lambda_{(T\omega,x)}(A\times C)&\text{if } \varphi(\omega)(0)=0\\
\lambda_{(T\omega,Sx)}(A\times C)&\text{if } \varphi(\omega)(0)=1.
\end{cases}
$$
We have
\begin{multline*}
\Phi\circ \widetilde{T} (\omega,y)(k_j)=C(k_j)\varphi(T\omega)(k_j)\\
 \iff S\circ \Phi(\omega,y)(k_j)=C(k_j)S\varphi(\omega)(k_j)\\
\iff \Phi(\omega,y)(k_j+1)=C(k_j)\varphi(\omega)(k_j+1).
\end{multline*}
Moreover, clearly $T\omega\in A \iff \omega \in T^{-1}A$. Finally, we also have
$$
|\{1\leq j\leq s : \varphi(\omega)(k_j+1)=0\}|=|\{1\leq j\leq s : \varphi(T\omega)(k_j)=0\}|.
$$
This ends the proof of~\eqref{eq:C} in view of the definition of measures $\lambda_{(\omega,y)}$. Therefore, for $\rho\in\mathcal{P}^e(\widetilde{T},\Omega\times \{0,1\}^\Z)$, we have
$$
\widetilde{\rho}:=\int \lambda_{(\omega,y)}\ d\rho(\omega,y)\in \mathcal{P}(T\times S,\Omega\times \{0,1\}^\Z)\text{ with }\Psi_\ast\widetilde{\rho}=\rho.
$$

The last step in the proof is to notice that
$$
M\circ (\varphi \times id_{\{0,1\}^\Z})=\Phi\circ \Psi.
$$
It follows that for any $\nu\in\mathcal{P}(S,Y)$\footnote{If $\widetilde{\Theta_\ast \nu}$ is not ergodic, we consider its ergodic decomposition and replace $\widetilde{\Theta_\ast\nu}$ with any of the ergodic components.} we have
$$
\nu=\Phi_\ast \Theta_\ast \nu=\Phi_\ast \Psi_\ast \widetilde{\Theta_\ast \nu}
=M_\ast (\varphi\times id_{\{0,1\}^\Z})_\ast \widetilde{\Theta_\ast \nu},
$$
which completes the proof as $\widetilde{\Theta_\ast \nu}\in\mathcal{P}( T\times S,\Omega\times \{0,1\}^\Z)$ and
$$
\varphi\times id_{\{0,1\}^\Z} \colon \Omega\times \{0,1\}^\Z \to X_\eta \times \{0,1\}^\Z
$$
induces an isomorphism between $\mathcal{P}(T\times S,\Omega\times \{0,1\}^\Z)$ and the simplex of probability $S\times S$-invariant measures on $X_\eta\times \{0,1\}^\Z$ whose projection onto the first coordinate is $\nu_\mathscr{B}$.
\end{proof}

We will show later, see Section~\ref{descrfinal}, that Theorem~\ref{twY} is valid for each member of $\cp^e(S,X_\eta)$ (with $\nu_{\mathscr{B}}$ replaced by a Mirsky measure of a subsystem). We postpone the proof of that fact to see first some introductory concepts and examples for a better understanding of the final result and its consequences.

\begin{Remark}
The language introduced in the course of the proof of Theorem~\ref{twY} can be used to provide another proof of Theorem~\ref{benji}. This proof is a simplification of the one presented in~\cite{Pec}.
\end{Remark}
\begin{proof}[Proof of Theorem~\ref{benji}]
Consider the transformation $\widetilde{T}_{C\times\{0,1\}^\Z}$ obtained by inducing $\widetilde{T}$ on the set $C\times \{0,1\}^\Z$. Notice that each point from $C\cap \Omega_0$ returns to $C\cap \Omega_0$ via $T$.  In other words, the induced map on $C\times \{0,1\}^\Z$ is well-defined up to a set of measure zero for any measure $\nu\in\mathcal{P}(\widetilde{T},\Omega\times \{0,1\}^\Z)$, since
$$
\nu(C\times\{0,1\}^\Z \cap \Omega\times\{0,1\}^\Z)=\PP(C\cap \Omega_0)=\PP(C)=\nu(C\times\{0,1\}^\Z).
$$
Moreover (see~\cite{Pec}), $\widetilde{T}_{C\times \{0,1\}^\Z}$ is a product transformation almost everywhere, with respect to any invariant measure. Since the first coordinate of $\widetilde{T}_{C\times \{0,1\}^\Z}$, i.e.\ $T_C$, is a uniquely ergodic map of zero entropy, it follows that $\widetilde{T}_{C\times\{0,1\}^\Z}$ is intrinsically ergodic, with topological entropy equal to $\log 2$. Therefore $\widetilde{T}$ is also intrinsically ergodic, with topological entropy equal to $\PP(C) \log 2>0$. Moreover, it follows from~\eqref{iden} that, in particular, $\Theta$ is 1-1. Hence,  $\Theta_\ast \colon\mathcal{P}(S,Y)\to\mathcal{P}(\widetilde{T},\Omega\times\{0,1\}^\Z)$ is also 1-1 and for any $\nu\in\mathcal{P}(S,Y)$, $h_{\nu}(S,Y)=h_{\Theta_\ast\nu}(\widetilde{T},\Omega\times\{0,1\}^\Z)$. The result follows now from Lemma~\ref{l3}.
\end{proof}

\subsubsection{Product type measures supported on $Y$}\label{31}

An important subset of joining type measures are product type measures which are ``ordinary convolutions'', see footnote~\ref{stopkac}. In this section, we will deal with measures of the form
$$
\nu_{\mathscr{B}}\ast\kappa:=M_\ast(\nu_{\mathscr{B}}\ot \kappa)\in\mathcal{P}^e(S,X_\eta).
$$
Clearly, whenever $\kappa\in\mathcal{P}^e(S,\{0,1\}^{\Z})$ is such that $(S,\{0,1\}^{\Z},\kappa)$ has no eigenvalue which is a $b_k$-root of unity (for all $k\geq1$) then  $\nu_{\mathscr{B}}\ast \kappa$ is ergodic. We will give now a condition on $\kappa$ which implies that the corresponding product type measure $\nu_{\mathscr{B}}\ast \kappa$ is supported on $Y$:
\begin{Prop}\label{hehe}
Suppose that for any natural numbers $t_1<t_2<\ldots$, the measure $\kappa\in\mathcal{P}^e(S,\{0,1\}^\Z)$ satisfies the following condition:
\beq\label{cond11}
\kappa\big(\{v\in\{0,1\}^{\Z}: v(t_1)=v(t_2)=\ldots =0\}\big)=0.
\eeq
Then $(\nu_{\mathscr{B}}\ast \kappa)(Y)=1$.
\end{Prop}

\begin{proof}
We have
\begin{multline*}
(\nu_{\mathscr{B}}\ast \kappa)(Y)=(\nu_{\mathscr{B}}\otimes \kappa)(M^{-1}(Y))=\PP\otimes \kappa((\varphi\times Id)^{-1}M^{-1}(Y))\\
=\PP\otimes \kappa((\Omega_0\times \{0,1\}^\Z)\cap((\varphi\times Id)^{-1}M^{-1}(Y))).
\end{multline*}
Moreover, for each $\omega\in \Omega_0$, we have $-\omega(k)\notin \text{supp}(\va(\omega))$ mod $b_k$, so the more, for each $u\in\{0,1\}^{\Z}$,
$-\omega(k)\notin \text{supp}(\va(\omega))\cdot u$ mod $b_k$. On the other hand, by Remark~\ref{ruz1} (see~\eqref{niesk}), if $z\in\Z/b_k\Z\setminus\{-\omega(k)\}$ then there is an infinite sequence $s_1<s_2<\ldots$ such that $\va(\omega)(-z+s_ib_k)=1$ for each $i\geq1$. Therefore, in view of~\eqref{cond11}, for $\kappa$-a.e.\ $u\in\{0,1\}^{\Z}$ there is $i_0=i_0(u)$ such that $u(-z+s_{i_0}b_k)=1$. Hence $(\va(\omega)\cdot u)(-z+s_{i_0}b_k)=1$, whence $M\circ (\varphi\times Id)(\omega,u)\in Y$. The result follows by Fubini's theorem.
\end{proof}
\begin{Remark}
Notice that each Bernoulli measure $B(p,1-p)$ satisfies condition~\eqref{cond11}. More generally, condition~\eqref{cond11} will be satisfied in each system $(S,\{0,1\}^{\Z},\kappa)$ which is mixing of all orders.
\end{Remark}

\subsubsection{Disintegration of product type measures on $Y$}
Let ${\cal L}_k$ be the family of blocks occurring on $X_\eta$ at $[-k,k]$. Fix $C\in{\cal L}_k$. Then
\begin{align*}
\nu_{\mathscr{B}}\ast\kappa(C)&=\nu_{\mathscr{B}}\ot\kappa(M^{-1}(C))\\
&=\nu_{\mathscr{B}}\ot\kappa\big(\{(x,z)\in X_\eta\times \{0,1\}^{\Z} : xz\in C\}\big)\\
&=\int_{X_\eta}\kappa(x^{-1}C)\,d\nu_{\mathscr{B}}(x)=\int_{\Omega_0}\kappa(\va(\omega)^{-1}C)\,d\PP(w),
\end{align*}
where $\va(\omega)^{-1}C:=\{D\in {\cal L}_k : \va(\omega)\cdot D=C\}$. Note that $\kappa(\va(\omega)^{-1}C)>0$ only if $C\leq \va(\omega)[-k,k]$. Moreover, whenever $\va(\omega)(s)=0$ then at the $s$th position of $D$ we can have $0$ or $1$.
It follows that
\beq\label{disint}
\nu_{\mathscr{B}}\ast \kappa=\int_{\Omega_0} \widetilde{\kappa}_\omega\,d\PP(\omega),
\eeq
where
\beq\label{disint1}
\widetilde{\kappa}_\omega(C)=\sum_{D\in{\cal L}_k: \va(\omega)\cdot D=C}\kappa(D).
\eeq
\begin{Remark}
Notice that in order to conclude that~\eqref{disint} represents a disintegration of $\nu_{\mathscr{B}}\ast\kappa$ over $\PP$, we need to know that $(T,\Omega,\PP)$ is a factor (via $\theta$) of the system determined by the convolution measure. For this it suffices that $(\nu_{\mathscr{B}}\ast \nu)(Y)=1$, see Proposition~\ref{hehe}.
\end{Remark}

\subsubsection{Product type measures on $Y$ isomorphic to direct products}\label{prody}

\begin{Remark}\label{uwa}
Note that~\eqref{disint1} says that if we want to see the distribution of $\widetilde{\kappa}_\omega$ on blocks, we need to look at the distribution of $\kappa$ on the cylinder sets
$$
C^{i_1,\ldots,i_m}_{j_1,\ldots,j_m}, \;\;i_r\in\{0,1\},
$$
where $-k\leq j_1<\ldots<j_m\leq k$ are all positions $t$ at which $\va(\omega)(t)=1$ and we copy this distribution to the family of all blocks smaller than or equal to $\va(\omega)[-k,k]$.
Notice that if $\kappa$ is a Bernoulli measure, we can ``squeeze'' (cf.\ Section~\ref{joa}) these positions and take the Bernoulli distribution on blocks of length~$m$ (in other words, we change $1$ to $0$ with probability $1-p$ when $\kappa=B(p,1-p)$). In particular, when $\kappa=B(\nicefrac12,\nicefrac12)$, we can see that $\widetilde{\kappa}_\omega=\mu_\omega$, where $\mu_\omega$ is as in Section~\ref{outline}, i.e.
\begin{align*}
&\mbox{the measure of maximal entropy for $(X_\eta,S)$}\\
&\mbox{is of product type: } \nu_{\mathscr{B}}\ast B(\nicefrac12,\nicefrac12).
\end{align*}
\end{Remark}

\begin{Prop}[cf.\ \cite{Pec} for the square-free system] \label{sspectrum2}
Let $\nu\in\mathcal{P}(S,X_\eta)$ be the measure of maximal entropy. Then $(S,X_\eta,\nu)$ is isomorphic to the direct product $(T,\Omega,\PP)\times (R,Z,{\cal D},\rho)$, where $R$ is a Bernoulli automorphism with entropy $\log2\cdot\Pi_{i=1}^\infty\left(1-\frac1{b_i}\right)$.\end{Prop}
\begin{proof}
By Remark~\ref{uwa}, $\nu=\nu_{\mathscr{B}}\ast B(\nicefrac12,\nicefrac12)$, so we have the following sequence of factors maps
$$
\left(S\times S,X_\eta\times \{0,1\}^{\Z},\nu_{\mathscr{B}}\ot B(\nicefrac12,\nicefrac12)\right)\stackrel{M}{\to} (S,X_\eta,\nu)\stackrel{\theta}{\to} (T,\Omega,\PP)\stackrel{\varphi}{\to} (S,X_\eta,\nu_{\mathscr{B}})
$$
with the last one being an isomorphism. Now,
$$
\left(S\times S,X_\eta\times \{0,1\}^{\Z},\nu_{\mathscr{B}}\ot B(\nicefrac12,\nicefrac12)\right)\stackrel{M\circ \theta\circ \varphi}{\longrightarrow} (S,X_\eta,\nu_{\mathscr{B}})$$ is relatively Bernoulli, so by Thouvenot's relative Bernoulli theory \cite{Th}, also
$$
(S,X_\eta,\nu)\stackrel{\theta\circ \varphi}{\longrightarrow} (S,X_\eta,\nu_{\mathscr{B}})
$$
is relatively Bernoulli, in other words the factor $ (S,X_\eta,\nu_{\mathscr{B}})$ splits off.\end{proof}

Consider now the case $\kappa=B(p,1-p)$, $0<p<1$, i.e.\ $\kappa$ is a Bernoulli measure. Fix $\omega\in \Omega$. By Remark~\ref{uwa}, for the Bernoulli measures, we have
\beq\label{dist1}
{\rm dist}_{\widetilde{\kappa}_\omega}\left(\bigvee_{j=0}^{n-1}S^jQ\right)=
{\rm dist}_{\kappa}\left(\bigvee_{\ell=0}^{m(\omega)-1}S^\ell Q\right),\eeq where
$
m(\omega):=|\{0\leq k\leq n-1:\va(\omega)(k)=1\}|.$
Hence, by~\eqref{dist1} and independence,
$$
\frac1n H_{\widetilde{\kappa}_\omega}\left(\bigvee_{j=0}^{n-1}S^jQ\right)=
\frac1n H_\kappa\left(\bigvee_{\ell=0}^{m(\omega)-1}S^\ell Q\right)=\frac{m(\omega)}n H_\kappa(Q).$$
It follows that for $\PP$-a.e.\ $\omega\in \Omega$, $\lim_{n\to\infty}\frac1n H_{\widetilde{\kappa}_\omega}\left(\bigvee_{j=0}^{n-1}S^jQ\right)=
\nu_{\mathscr{B}}(C^1_0)H_{\kappa}(Q)$.  Since $h_{\nu_{\mathscr{B}}\ast\kappa}(S,Q)$ is equal to the relative entropy with respect to the $(T,\Omega,\PP)$ factor (as the latter has zero entropy),
$$
h_{\nu_{\mathscr{B}}\ast\kappa}(S,Q)=\lim_{n\to\infty}\frac1n\int_{\Omega}H_{\widetilde
{\kappa}_\omega}\left(\bigvee_{j=0}^{n-1}S^jQ\right)\,d\PP(\omega),$$
and we obtain the following result:

\begin{Prop}\label{hBp} If $\kappa=B(p,1-p)$ then
$$
h(S,X_\eta,\nu_{\mathscr{B}}\ast B(p,1-p))=-(p\log p+(1-p)\log(1-p))\Pi_{i\geq1}(1-1/b_i).
$$
\end{Prop}
\begin{Remark}
It follows by the above that:
\begin{itemize}
\item For each value $0\leq h\leq\log2\cdot \Pi_{i\geq1}(1-1/b_i)$ there is an ergodic measure $\kappa$ such that
$h(S,X_\eta,\nu_{\mathscr{B}}\ast\kappa)=h$.
\item Similarly as in the case $\kappa=B(\nicefrac12,\nicefrac12)$, cf.\ Proposition~\ref{sspectrum2}, we obtain that the dynamical system $(S,X_\eta,\nu_{\mathscr{B}}\ast B(p,1-p))$ is isomorphic to the direct product of $(T,\Omega,\PP)$ and a Bernoulli automorphism with the entropy $-(p\log p+(1-p)\log(1-p))\Pi_{i\geq1}(1-1/b_i)$.
\end{itemize}
\end{Remark}

\begin{Question}
Can we obtain a general entropy formula for the product type measures $\nu_{\mathscr{B}}\ast\kappa$, e.g.\ where $\kappa$ satisfies~\eqref{cond11}? Is it true that entropy of the product  type measure is positive whenever the entropy of $\kappa$ is positive?
Is the entropy of $\nu_{\mathscr{B}}\ast \kappa$ always smaller than the entropy of $\kappa$ provided  that the entropy of $\kappa$ is positive?
\end{Question}

\begin{Remark}
Notice that except for the situation when $\kappa=\delta_{(\ldots11\ldots)}$, the map $\theta\colon(S,X_\eta,\nu_{\mathscr{B}}\ast\kappa)\to(T,\Omega,\PP)$ cannot be an isomorphism. Indeed, if so then the conditional measures are Dirac measures, and in particular the distribution of $\widetilde{\kappa}_\omega$ on blocks of length~$1$ must be trivial. However this distribution is given by the distribution of $\kappa$ on blocks of length~$1$ which cannot be trivial if $\kappa\neq\delta_{(\ldots11\ldots)}$. Therefore, if $\kappa$ yields a K-automorphism, then
$(S,X_\eta,\nu_{\mathscr{B}}\ast\kappa)\to (T,\Omega,\PP)$ is relatively K, hence the entropy of $\nu_{\mathscr{B}}\ast\kappa$ is positive.
\end{Remark}

\subsection{Invariant measures on $X_\eta$}
\subsubsection{Zero entropy measures and filtering}\label{irrational}
As we have seen in Section~\ref{prody}, if $\kappa=B(p,1-p)$ then
$$
h_{\nu_{\mathscr{B}}\otimes \kappa}(S\times S,X_\eta\otimes \{0,1\}^\Z)>h_{\nu_{\mathscr{B}}\ast \kappa}(S,X_\eta).
$$
In particular, the map $M$ cannot be an isomorphism. Clearly, if $\kappa=\delta_{(\ldots11\ldots)}$ then $\nu_{\mathscr{B}}\ast \kappa=\nu_{\mathscr{B}}$, i.e.\ $M$ {\bf is} an isomorphism. A general question arises whether $M$ can be an isomorphism of $(S\times S,X_\eta\times \{0,1\}^\Z,\nu_{\mathscr{B}}\otimes \kappa)$ and $(S,X_\eta,\nu_{\mathscr{B}}\ast \kappa)$ for $\kappa\neq\delta_{(\ldots11\ldots)}$. In particular, we will be interested in the situation when $\kappa$ yields a zero entropy system.

Now, we will look and the product type measures  from the point of view of the filtering problem in ergodic theory  (\cite{Bu-Le-Le}, \cite{Fu}, \cite{Fu-Pe-We}). For this, we will need some notation (partially borrowed from~\cite{Ab-Le-Ru}) and some tools.

Let $C\subset \Omega$ be given by $C:=\varphi^{-1}(C_0^1)$, where $C_0^1=\{x\in X_\eta : x(0)=1\}$, i.e.\ $C=\{\omega\in\Omega:(\forall k\geq1)\;\;\omega(k)\neq0\}$. Then
\beq\label{ejm1}
\va(\omega)=(f(T^n\omega))_{n\in\Z},
\eeq
where $f(\omega)=\raz_C(\omega)$.
\begin{Lemma}\label{l-jm} The partition $\{C,\Omega\setminus C\}$
is a generating partition.
\end{Lemma}
\begin{proof}
This is just a reformulation of the fact the $\va$ is $\PP$-a.e.\ 1-1 (see Lemma~\ref{luz1}).
\end{proof}
Recall also that
\beq\label{ejm2}
\nu_{\mathscr{B}}\ast\kappa=(M\circ(\va\times Id))_\ast(\PP\ot \kappa).\eeq
Furthermore, for each $\omega\in\Omega$,  $z\in\{0,1\}^{\Z}$ and $n\in\Z$, we have
\begin{multline*}
M\circ(\va\times Id)(\omega,z)(n)=\va(\omega)(n)\cdot z(n)\\
=\raz_C(T^n\omega)\cdot\raz_{C^1_0}(S^nz)=
\raz_{C\times C^1_0}((T\times S)^n(\omega,z)).
\end{multline*}
Now, if $p_0: X_\eta\to\{0,1\}$ denotes the projection on the zero coordinate,
$$
\raz_{C\times C^1_0}=\raz_C\ot\raz_{C^1_0}=p_0\circ M\circ (\va\times Id).$$
Let
$
\cg:=(M\circ(\va\times Id))^{-1}(\cb(X_\eta)).
$
It follows that the set
\beq\label{ejm3}
\mbox{$C\times C^1_0$ is $\cg$-measurable.}\eeq
Moreover,
\begin{align}
\begin{split}\label{ejm4}
&\mbox{$(S,X_\eta,\nu_{\mathscr{B}}\ast\kappa)$ is measure-theoretic isomorphic to}\\
&\mbox{$(T\times S,\Omega\times \{0,1\}^{\Z}/\cg,\cg,\PP\ot\kappa)$.}
\end{split}
\end{align}

We will also need some ergodic theory results coming from \cite{Bu-Le-Le}, concerning the filtering problem.
The following result can be proved by repeating almost verbatim the proof of Proposition~5 in \cite{Bu-Le-Le}.

\begin{Prop}[cf.\ \cite{Bu-Le-Le}]\label{p-jm} Assume that $T$ and $S$ are ergodic automorphisms of probability standard Borel spaces $\xbm$ and $\ycn$, respectively.
Assume that for each ergodic self-joinings $\la$ of $T$ and $\rho$ of $S$, we have
\beq\label{ejm5}
(T\times T, X\times X,\la)\perp (S\times S,Y\times Y,\rho).\footnote{We write $\perp$ between two measure-theoretic automorphisms if they are disjoint, i.e.\ if the only joining between them is product measure~\cite{Gl}.}
\eeq
Assume that $\cf\subset\cb\ot\cc$ is a factor of $(T\times S,X\times Y,\mu\ot\nu)$.
Then there exist factors $\cb_1\subset\cb$, $\cc_1\subset\cc$ and compact subgroups
$\ch\subset C(T|_{\cb_1})$,\footnote{We denote the action of $T$ on the factor $(X/\cb_1,\cb_1,\mu|_{\cb_1})$ by $T|_{\cb_1}$. Given an automorphism $T$, $C(T)$ stands for its centralizer.} $\ch'\subset C(S|_{\cc_1})$ with a continuous group isomorphism $\ch\ni W\mapsto W'\in\ch'$
such that
\beq\label{ejm6}
\cf={\rm Fix}\big(\{W\times W' : W\in \ch\}\big).\footnote{Given an automorphism $T$ acting on $\xbm$ and $\ch\subset C(T)$, we
set $$ {\rm Fix}(\ch):=\{A\in\cb : (\forall W\in \ch)\;\;WA=A\}. $$ Clearly, ${\rm Fix}(\ch)$ is a factor of $T$.}
\eeq
In particular,
\beq\label{ejm7}\cf\supset {\rm Fix}(\ch)\ot{\rm Fix}(\ch').\eeq
\end{Prop}

\begin{Cor}\label{c-jm}Under the assumptions of Proposition~\ref{p-jm}, suppose additionally that
$\cf$ contains a ``rectangle'' $C\times J\in\cf$, where the partitions $\{C,X\setminus C\}$, $\{J,Y\setminus J\}$ are generating for $T$ and $S$, respectively.
Then $\cf=\cb\ot\cc$.
\end{Cor}
\begin{proof}
The set $C\times J$ is fixed by all elements $W\times W'$, $W\in\ch$, whence $C\in{\rm Fix}(\ch)$ and $J\in {\rm Fix}(\ch')$.
Hence $C\times J\in {\rm Fix}(\ch)\ot{\rm Fix}(\ch')$. The latter factor is a product factor, so it is invariant under the product $\Z^2$-action $\{T^m\times S^n:m,n\in\Z\}$, i.e.\ $(T^m\times S^n)(C\times J)\in {\rm Fix}(\ch)\ot{\rm Fix}(\ch')$ for each $m,n\in\Z$, and the result follows.
\end{proof}

Now, consider  $\kappa\in \cp^e(S,\{0,1\}^\Z)$ such that the following holds:
\begin{align}\label{toterg}
\begin{split}
&\text{Every ergodic self-joining $\rho$ of $(S,\{0,1\}^{\Z},\kappa)$ yields an ergodic }\\
&\text{system which has no $b_1\cdot\ldots\cdot b_k$-root of unity, $k\geq1$, in its spectrum.}
\end{split}
\end{align}
(For example, if each such joining is totally ergodic, then~\eqref{toterg} can be applied to an arbitrary $\mathscr{B}$-free system.) We recall that~\eqref{toterg} forces $\kappa$ to have zero entropy (by Smorodinsky-Thouvenot's theorem \cite{Sm-Th}).

Since every ergodic self-joining of $(T,\Omega,\PP)$ is a graph joining,~\eqref{toterg} yields \eqref{ejm5} for the relevant systems. Note also that~\eqref{toterg} is the double disjointness condition of $(S,\{0,1\}^{\Z},\kappa)$ with $(T,\Omega,\PP)$ from \cite{Fu-Pe-We}. Thus, we have shown the following:
\begin{Cor}\label{c-jm1} If $\kappa$ satisfies~\eqref{toterg} then $(S,X_\eta,\nu_{\mathscr{B}}\ast\kappa)$ is isomorphic to the  Cartesian product $(T\times S,X_\eta\times \{0,1\}^{\Z},\nu_{\mathscr{B}}\otimes\kappa)$.
\end{Cor}

\begin{Remark} If $(S,\{0,1\}^{\Z},\kappa)$ represents an irrational rotation,~\eqref{toterg} is clearly satisfied, but there are many
weakly mixing systems satisfying~\eqref{toterg}, e.g.: Gaussian systems GAG \cite{Le-Pa-Th}, simple systems \cite{Gl} and factors of such systems, in particular, horocycle flows \cite{Th1}.
\end{Remark}

\begin{Remark}\label{filterB}
{We will give now a direct proof of the fact that whenever $\kappa$ represents an irrational rotation then we can filter out both coordinates, i.e.\ $M$ is an isomorphism. Indeed,  we take for $J\subset\T$ an interval. Then the rectangle $C\times J$ is in the smallest invariant $\sigma$-algebra $\cg$ which makes the map
$$
(\omega,z)\mapsto (\raz_{C\times J}((T\times R_\alpha)^n(\omega,z)))_{n\in\Z}
$$
measurable (cf.\ \eqref{ejm3}). Given $\vep>0$ we can $\vep$-approximate, whenever $k\geq1$ is large enough, the set $C$ by the levels of a $T$-tower (unique up to cyclic permutation of the levels) of height $M_k:=b_1\cdot\ldots\cdot b_k$  which fulfills the whole space. If we fix such a $k$ and take any $1\leq m<M_k$  then we can find a sequence $(n_i)_{i\geq1}$ such that
\beq\label{irr0}
n_i=m\;\;{\rm mod} \;\;M_k\text{ and }n_i\alpha\to 0.
\eeq
Indeed, this is a consequence of the fact that $(\ell M_k+m)\alpha$ is close to zero if and only if $\ell(M_k\alpha)$ is close to $-m\alpha$ and the rotation by $M_k\alpha$ is minimal. Since $(T^{n_i}\times S^{n_i})(C\times J)\in\cg$, it easily follows by~\eqref{irr0} that $\Omega\times J\in\cg$, which means that we can filter out the second coordinate.

In order to obtain $C\times\T\in\cg$ we proceed as follows. We $\vep$-approximate the set $C$ by the levels of the tower of height $M_k$. Now, consider $R_{M_k\alpha}$. Since $J$ is an interval, we
can find $n_1<\ldots<n_r$, so that $\la_{\T}(\bigcup_{j=1}^rR_{M_k\alpha}^{n_j}J)>1-\vep$; here it is important that $r$ depends only on $J$ and not on $M_k\alpha$, $r$ is ``comparable'' with $1/|J|$. Since $r$ is fixed, we can easily see that whatever the numbers $n_1<\ldots<n_r$ are, the set
$\bigcap_{j=1}^rT^{n_jM_k}C$ will be $\vep$-close to $C$. In this way, we obtain that $C\times \T\in\cg$ and hence $M$ is an isomorphism.}
\end{Remark}

\subsubsection{Rational discrete spectrum}\label{31}
We begin this section with two examples, showing the basic relations between the Mirsky measures for various free systems, under some additional assumptions on the sequences determining these systems.

\begin{Example}\label{przyk1}
Let $X_{\eta}$ and $X_{\eta'}$ be two free systems, with $\mathscr{B}=\{b_k : k\geq1\}$ and
$\mathscr{B}'=\{b'_k : k\geq1\}$ respectively, and assume that $b'_k|b_k$ for each $k\geq1$.
Then clearly $\eta'\leq \eta$. In particular, each block that occurs on $\eta'$ is dominated by a block that occurs on $\eta$, whence
\begin{equation}
\label{zaw} X_{\eta'}\subset X_{\eta}.
\end{equation}
Therefore, $\nu_{\mathscr{B}'}\in \mathcal{P}^e(S,X_\eta)$ is a measure which yields a dynamical system whose spectrum is ``incomplete'' in the sense that it is smaller than the whole group of $b_k$-roots of unity, $k\geq1$.
\end{Example}

\begin{Example}\label{przyk2}
Now, assume that $\mathscr{B}=\{b_k : k\geq1\}$ is a free system and take a subset $\ov{\mathscr{B}}=\{\ov b_k:\;k\geq1\}$ with $\ov b_{k}=b_{n_k}$. It follows
that
\beq\label{zaw1} X_{\eta}\subset X_{\ov\eta}.\eeq
Now, we observe a different phenomenon than in Example~\ref{przyk1}. A larger $\ov{\mathscr{B}}$-free system has an invariant measure, namely the Mirsky measure of $(S,X_\eta)$, which yields a system whose spectrum is larger than the ``expected'' one. In fact, the larger system has a smaller underlying odometer: $(\ov T,\ov\Omega,\ov{\PP})$ is a factor of $(T,\Omega,\PP)$.
\end{Example}

\begin{Remark}
In view of the above two examples, one might expect that the condition that $X_\eta\subset X_{\eta'}$ can be expressed in terms of some relation between the sets $\mathscr{B}$ and $\mathscr{B}'$. This is indeed the case, see Proposition~\ref{appl1} and~\ref{appl2} for a complete charaterization.
\end{Remark}

\begin{Prop}\label{pspectrum1}
Assume that $\mathcal{P}^e(S,X_\eta)\ni\nu\neq \delta_{(\ldots,0,0,\ldots)}$.
The dynamical system $(S,X_\eta,\nu)$ has  an infinite rational discrete spectrum. More precisely, the discrete spectrum part contains, for each $k\geq1$, all $b'_k$-roots of unity for some $1<b'_k|b_k$.\footnote{Notice that $(b'_k,b'_\ell)=1$ whenever $k\neq\ell$ by~\eqref{f1}.}
\end{Prop}
In order to prove the above proposition, we will use a refinement of the approach taken in~\cite{Pec}. Let us introduce first some notation which will be also used later. Fix $\delta_{(\ldots,0,0,\ldots)}\neq \nu\in \mathcal{P}^e(S,X_\eta)$. Given $k\geq1$ and $1\leq s_k\leq b_k-1$, set
$$
Y_{k,s_k}:=\{x\in X_\eta: |{\rm supp}(x)\;{\rm mod}\;b_k|=b_k-s_k\}.
$$
Then $Y_{k,s_k}$ is  Borel and $SY_{k,s_k}=Y_{k,s_k}$. By ergodicity, for each $k\geq 1$ there is exactly one $s_k$ such that $\nu(Y_{k,s_k})=1$.  Now, for
$a_i\in\Z/b_k\Z$, $i=1,\ldots,s_k$, with $a_i\neq a_j$ whenever $i\neq j$, we set
$$
Y_{k,s_k;a_1,\ldots,a_{s_k}}:=\{x\in X_\eta:
{\rm supp}(x)\;{\rm mod}\;b_k=\Z/b_k\Z\setminus\{a_1,\ldots,a_{s_k}\}\}\subset Y_{k,s_k}.
$$
For each $k\geq 1$, any two sets of such form are either disjoint or they coincide. Moreover, their union gives $Y_{k,s_k}$.
It follows that there exists $(a_1^k,\ldots,a_{s_k}^k)$ such that
$\nu(Y_{k,s_k;a_1^k,\ldots,a_{s_k}^k})>0$.
Since ${\rm supp}(Sx)={\rm supp}(x)-1$, we have
\begin{equation}\label{her11}
SY_{k,s_k;a_1^k,\ldots,a_{s_k}^k}=Y_{k,s_k;a_1^k-1,\ldots,a_{s_k}^k-1}.
\end{equation}
Let
\begin{equation}\label{primy}
b_k':=\min\{j\geq 1 : \{a_1^k,\dots,a_{s_k}^k\}=\{a_1^k-j,\dots,a_{s_k}^k-j\}\}
\end{equation}
and note that $b_k'\geq 2$. Clearly, $S^{b'_k}Y_{k,s_k;a_1^k,\ldots,a_{s_k}^k}=Y_{k,s_k;a_1^k,
\ldots,a_{s_k}^k}$ and the sets
$$
Y_{k,s_k;a_1^k,\ldots,a_{s_k}^k},SY_{k,s_k;a_1^k,\ldots,a_{s_k}^k},\ldots,
S^{b'_k-1}Y_{k,s_k;a_1^k,\ldots,a_{s_k}^k}
$$
are pairwise disjoint. Moreover, by ergodicity,
$$
\nu\Big(\bigcup_{j=0}^{b_k'-1}S^{j} Y_{k,s_k;a_1^k,\ldots,a_{s_k}^k}\Big)=1.
$$
Since
$S^{b_k}Y_{k,s_k;a_1^k,\ldots,a_{s_k}^k}=Y_{k,s_k;a_1^k, \ldots,a_{s_k}^k}$, we have $b'_k|b_k$. Finally, for $\underline{s}=(s_k)_{k\geq 1}$, we set $Y_{\underline{s}}:=\bigcap_{k\geq 1}Y_{k,s_k}$.

\begin{proof}[Proof of Proposition~\ref{pspectrum1}]
It suffices to notice that for $1\leq s_k\leq b_k-1$ and $\{a_1^k,\dots,a_{s_k}^k\}$ chosen so that $\nu(Y_{k,s_k;a_1^k,\dots,a_{s_k}^k})>0$,
and $b_k'$ given by~\eqref{primy}, the partition of $Y_{k,s_k}$ into sets
$$
S^j Y_{k,s_k;a_1^k,\ldots,a_{s_k}^k},\ 0\leq j\leq b_k'-1
$$
is a Rokhlin tower fulfilling the whole space, whence the $b'_k$-root of unity is an eigenvalue of $(S,X_\eta,\nu)$.
\end{proof}

We will give now another proof of Proposition~\ref{pspectrum1}. For this, we will need the following lemma:
\begin{Lemma}\label{lspectrum1} Assume that $\mathcal{P}^e(S,X_\eta)\ni\nu\neq \delta_{(\ldots,0,0,\ldots)}$. Denote by $R_k$ the rotation $z\mapsto z+1$ on $\Z/b_k\Z$ (considered as an ergodic system). Then $(S,X_\eta,\nu)$ is not disjoint with $R_k$.
\end{Lemma}
\begin{proof}  Suppose that $(S,X_\eta,\nu)$ is disjoint (see \cite{Fu}, \cite{Gl}) with $R_k$.
Let $y\in X_\eta$ be a generic point for $\nu$. Since $y$ is admissible, we can pick $a_k\in\Z/b_k\Z$ which does not belong to the support of $y$ mod $b_k$. Let $z\in\{0,1\}^{\Z}$ be such that
$$
z(n)=0 \iff n=a_k\bmod b_k.
$$
 This point is clearly generic
for the periodic measure $\Delta_k:=\frac1{b_k}\sum_{j=0}^{b_k-1}\delta_{S^jz}$ and the resulting dynamical system is isomorphic to $R_k$. Moreover, since $y(a_k+\ell b_k)=0$ for each $\ell\in\Z$,
\beq\label{spect1}
y\leq z.\eeq
By the disjointness assumption, $(y,z)\in X_\eta\times\{0,1\}^{\Z}$ is a generic point for the product measure $\nu\ot\Delta_k$. But $\nu(C^1_0)>0$ (since $\nu\neq\delta_{(\ldots,0,0,\ldots)}$) and $\Delta_k(C^0_0)>0$ so the product measure of $C^1_0\times C^0_0$ is positive while, by~\eqref{spect1}, no point $(S^iy,S^iz)$ belongs to $C^1_0\times C^0_0$, a contradiction.
\end{proof}
\begin{proof}[Second proof of Proposition~\ref{pspectrum1}]
It follows from Lemma~\ref{lspectrum1} that for each $k\geq1$ we have no disjointness of $(S,X_\eta,\nu)$ with $R_k$. This means that $(S,X_\eta,\nu)$ must have, for each $k\geq1$, a nontrivial common factor with $R_k$, equivalently a common nontrivial eigenvalue.
\end{proof}

\begin{Remark}
Consider $b_k=p_k^2$, $k\geq1$ and then the corresponding square free system. By Proposition~\ref{pspectrum1}, any nontrivial ergodic measure must have at  least all $p_k$-roots of unity in the spectrum of the corresponding dynamical system. A natural question arises whether there is a measure which yields the dynamical system with precisely such a spectrum.\footnote{Notice that this question cannot be answered following the path taken in Example~\ref{przyk1} since $\sum_{k\geq 1}1/p_k =\infty$.} We will show later that such a measure cannot exist, see Corollary~\ref{nomeasure}.
\end{Remark}
In connection with the above remark, we consider the following example:
\begin{Example}
Let $\mathscr{B}=\{p_{i_k}\in\mathscr{P}:k\geq1\}$, so that $\sum_{k\geq1}1/p_{i_k}<+\infty$. Then
$$
\mathscr{P}\setminus \mathscr{B}=\{q_i: i\geq1\}.$$
Let $\kappa\in\cp^e(S,\{0,1\}^{\Z})$ be such that $(S,\{0,1\}^{\Z},\kappa)$ has discrete spectrum with the group of eigenvalues equal to the $q_1\cdot\ldots\cdot q_i$-roots of unity, $i\geq1$ (such $\kappa$ exists by Krieger's theorem \cite{Gl}). Now, $(S,\{0,1\}^{\Z},\kappa)$ has discrete spectrum, so each ergodic self-joining of it is a graph joining and therefore~\eqref{toterg} is satisfied. Now, by Corollary~\ref{c-jm1}, the measure $\rho:=\nu_{\mathscr{B}}\ast\kappa$ is such that the spectrum of $(S,\{0,1\}^\Z,\rho)$ is equal to all roots of unity of order $p_1\cdot\ldots\cdot p_k$, $k\geq 1$.
\end{Example}

\subsubsection{Filtering $\PP$ from $\nu_{\mathscr{B}}\ast\kappa$}
Recall that since we have an equivariant Borel map $\theta\colon Y\to\Omega$, it follows immediately that for any $\nu\in\mathcal{P}^e(S,Y)$ the corresponding dynamical system $(S,X_\eta,\nu)$ has $(T, \Omega,\PP)$ as its factor. A natural question arises whether each measure $\nu\in\mathcal{P}^e(S,X_\eta)$ such that the point spectrum of $(S,X_\eta,\nu)$ contains the $b_1\cdot\ldots \cdot b_k$-roots of unity, $k\geq1$ must be concentrated on~$Y$. We will see in Example~\ref{ex4} below that this is not the case.

\begin{Remark}\label{uwaga9}
Note that the Mirsky measure $\nu_{\mathscr{B}}$ is concentrated on
$Y=\bigcap_{k\geq 1}Y_{k,1}$. Assume that $1<b'_k|b_k$, $b_k/b'_k\geq2$,  $k\geq1$, so that $\mathscr{B}':=\{b'_k : k\geq1\}$ satisfies~\eqref{f1}. Then
\beq\label{conc}
\nu_{\mathscr{B}'}\Big(\bigcap_{k\geq1}Y_{k,s_k}\Big)=1,
\eeq
where $s_k\geq b_k/b'_k\geq2$, $k\geq1$. Indeed, if $a\notin {\rm supp}(y)\;{\rm mod}\;b'_k$ then
$a+jb'_k\notin {\rm supp}(y)\;{\rm mod}\; b_k$ for $j=0,1,\ldots, {b_k}/{b'_k}-1$. Moreover, $\nu_{\mathscr{B}'}(Y)=0$ since $\mathscr{B}'\neq\mathscr{B}$.
\end{Remark}
\begin{Remark}
Notice that the Mirsky measures $\nu_{\mathscr{B}'}$ in Example~\ref{przyk1} vanish on the set $Y=Y(X_\eta)$. Note however, that
$\nu_{\mathscr{B}}(\ov{Y})=1$ for $\ov{Y}=Y(X_{\ov{\eta}})$ in Example~\ref{przyk2}.
\end{Remark}

\begin{Example}\label{ex4}
Consider $\mathscr{B}=\{b_k : k\geq1\}$ a free system in which $b_k=b'_k\cdot c_k$, $(b'_k,c_k)=1$, $b_k'\geq 2$ for $k\geq1$,
$$
\sum_{k\geq1}{1}/{b'_k}<+\infty \text{ and }\sum_{k\geq1}1/{c_k}<+\infty.
$$
We then obtain two more free systems:
$$
\widetilde{\mathscr{B}}=\{b_1',c_1,b'_2,c_2,\ldots\}\;\;\mbox{and}\;\;
\mathscr{B}'=\{b_1',b'_2,\ldots\}.$$
Using \eqref{zaw}, \eqref{zaw1} and Remark~\ref{uwaga9}, we obtain
$$
X_{\widetilde{\mathscr{B}}}\subset X_{\mathscr{B}'}\subset \bigcap_{k\geq 1} Y_{k,s_k}(X_{\mathscr{B}})\subset X_{\mathscr{B}},$$ where
$s_k\geq2$ for $k\geq1$.  But the point spectra of
$(S,X_{\widetilde{\mathscr{B}}},\nu_{\widetilde{\mathscr{B}}})$ and $(S,X_{\mathscr{B}},\nu_{\mathscr{B}})$ are the same. Finally, $\nu_{\widetilde{\mathscr{B}}}(Y)=0$.
\end{Example}

Now, we give a condition on $\kappa$ which implies that the corresponding product type measure $\nu_\mathscr{B}\ast \kappa$ is such that $(T,\Omega)$ is a factor of $(S,X_\eta,\nu_{\mathscr{B}}\ast\kappa)$. It is unclear, whether this condition implies that $(\nu_{\mathscr{B}}\ast\kappa)(Y)=1$.

\begin{Prop}\label{filt1wsp}
If $\kappa\in\cp^e(S,\{0,1\}^{\Z})$  yields a totally ergodic system then $(S,X_\eta,\nu_{\mathscr{B}}\ast\kappa)$ has full rational spectrum, i.e.\ $(T,\Omega,\PP)$ is its factor.
\end{Prop}

\begin{proof}
We will proceed as in Remark~\ref{filterB}, detailing more on $C$ and the towers for the odometer $(T,\Omega,\PP)$ (which allows us to bypass the existence of $r$ in Remark~\ref{filterB}).

Assume that $(S,\{0,1\}^{\Z},\kappa)$ is totally ergodic and let $J:=C^1_0$. It follows from~\eqref{ejm3} that $C\times J\in\cg=(M\circ(\va\times Id))^{-1}(\cb(X_\eta))$. We will show that also $C\times\{0,1\}^{\Z}\in\cg$. For this aim, consider
\beq\label{filterM}
 \bigcup_{j=0}^{R-1} T^{jM_k}C\times S^{jM_k}J,
\eeq
where $M_k:=b_1\cdot\ldots\cdot b_k$ and $R\geq 1$. Consider a tower for $T$ of height $M_k$, with the set
$\{\omega\in\Omega: \omega(1)=\ldots=\omega(k)=0\}$ as the base. The levels of this tower are sets of the form  $\{\omega\in\Omega: \omega(1)=i_1,\ldots,\omega(k)=i_k\}$, i.e.\ they
are indexed by $(i_1,\ldots,i_k)\in\Z/b_1\Z\times\ldots\times \Z/b_k\Z$. Whenever the level $(i_1,\ldots,i_k)$ contains  $i_s=0$, it is disjoint with $C$.
If at all positions $(i_1,\ldots,i_k)$ we see non-zero values then $C$ is contained in such a level. More than that, we can compute the fraction of the level it occupies (which is smaller than  $\sum_{t=k+1}1/b_t$).

If $R$ is large enough then   $\kappa\left(\bigcup_{j=0}^{R-1} S^{jM_k}J\right)$ is as large as we want (since $S^{M_k}$ is ergodic). We need to show that the set~\eqref{filterM} is close to  $C\times\{0,1\}^{\Z}$. Indeed, we have\footnote{We use here the following: whenever $C,A_i\subset X$,  $D,B_i\subset Y$, we have $
\bigcup_{i=1}^R(A_i\times B_i) \triangle (C\times D)
                      \subset
\Big(\Big(\bigcup_{i=1}^R A_i  \triangle C\Big)\times Y\Big)  \cup  \Big(X\times \Big(\bigcup_{i=1}B_i \triangle D\Big)\Big).$}
\begin{multline*}
\Big(\bigcup_{j=0}^{R-1} T^{jM_k}C\times S^{jM_k}J\Big)
\triangle \Big(C\times\{0,1\}^{\Z}\Big)\\
\subset \Big(\Big(\bigcup_{j=0}^{R-1} T^{jM_k}C \setminus C\Big)\times \{0,1\}^{\Z}\Big) \cup
\Big(\Omega\times \Big(\{0,1\}^{\Z}\setminus \bigcup_{j=0}^{R-1}S^{jM_k}J\Big)\Big),
\end{multline*}
whence
\begin{multline*}
\PP\ot\kappa\Big(\Big(\bigcup_{j=0}^{R-1} T^{jM_k}C\times S^{jM_k}J\Big)
\triangle \Big(C\times\{0,1\}^{\Z}\Big)\Big)\\
\leq \PP\Big(\Big(\bigcup_{j=0}^{R-1} T^{jM_k}C \Big)\setminus C\Big)+\vep.
\end{multline*}
Moreover, since each $T^{jM_k}$ sends the level of the tower into itself, the levels that were disjoint with $C$ remain disjoint and the first summand above is not larger than the approximation of $C$ given by the union of levels containing $C$.
\end{proof}

\subsubsection{Ergodic invariant measures on $X_\eta$ are of joining type}\label{descrfinal}

In Section~\ref{joa}, we have proved that each measure $\nu\in\cp^e(S,Y)$ is of joining type, more precisely, $\nu=M_\ast(\widetilde{\rho})$, where $\widetilde{\rho}\in\cp^e(S\times S,X_\eta\times \{0,1\}^{\Z})$ satisfies $\widetilde{\rho}|_{X_\eta}=\nu_{\mathscr{B}}$. One could now expect that the converse also holds. That is, whenever we have $\widetilde{\rho}\in\mathcal{P}^e(S\times S,X_\eta\times\{0,1\}^\Z)$  which is an ergodic joining of $\nu_{\mathscr{B}}$ and $\kappa:=\widetilde{\rho}|_{\{0,1\}^{\Z}}$ then $M_\ast(\widetilde{\rho})\in\mathcal{P}(S,Y)$ (in particular, the corresponding dynamical system has ``full'' rational discrete spectrum).
This is however not true:
\begin{Example}\label{pr4}
Consider the situation, where $\mathscr{B}'=\{b_i':i\geq1\}$ yields a free systems, with $1<b'_i|b_i$, $i\geq1$. Let
$$
\pi:\Omega\to \Omega':=\Pi_{i=1}\Z/b'_i\Z,\;\;\pi((\omega(k))_{k\geq1})=
(\omega'(k))_{k\geq1},
$$
$\omega'(k)=\omega(k)$ mod~$b'_k$, $k\geq1$. Then $\pi$ is equivariant and $\pi_\ast(\PP)=\PP'$. The measure $M_\ast(\la)$, where $\la=\nu_{\mathscr{B}}\vee \nu_{\mathscr{B}'}$ stands for the diagonal embedding of $(X_{\eta'},\nu_{\mathscr{B}'})$ in $(X_\eta,\nu_{\mathscr{B}})$, is concentrated on the set
$$
W:=\{\va(\omega)\cdot\va'(\omega'): (\omega,\omega')\in\Omega\times\Omega', \pi(\omega)=\omega'\}.
$$
However, whenever $\pi(\omega)=\omega'$, we have $\va'(\omega')\leq\va(\omega)$. It follows that for each $n\in\Z$,
$\va(\omega)(n)\cdot\va'(\omega')(n)=\va'(\omega')(n)$ and therefore
$$
M_\ast (\nu_{\mathscr{B}}\vee\nu_{\mathscr{B}'})=\nu_{\mathscr{B}'}.
$$
\end{Example}

We will show now, how to extend Theorem~\ref{twY} to obtain Theorem~\ref{twYall}, thus providing a description of all invariant measures for $\mathscr{B}$-free systems. As a matter of fact,
the proof goes along the same lines as the proof of Theorem~\ref{twY}. However, to see that similar arguments are indeed valid, we need to define several objects. For $\underline{s}=(s_k)_{k\geq 1}$, $\underline{a}=(\underline{a}^k)_{k\geq 1}$, $\underline{a}^k=\{a_1^k,\dots,a_{s_k}^k\}$, cf.~\eqref{her11} and~\eqref{primy}, let
$$
Y_{\underline{s},\underline{a}}:=\bigcap_{k\geq 1}\Big(\bigcup_{j=0}^{b_k'-1}S^{j} Y_{k,s_k;a_1^k,\ldots,a_{s_k}^k}\Big).
$$
and
$$
\overline{Y}_{\underline{s},\underline{a}}:=\{x\in \{0,1\}^\Z : \text{ any block on }x\text{ occurs on }Y_{\underline{s},\underline{a}}\}.
$$
Notice that we can assume without loss of generality that $a_1^k=0$ for each $k\geq 1$. Moreover, since the sets $Y_\sa$ are Borel and shift-invariant, for each measure $\nu\in\mathcal{P}^e(S,X_\eta)$, there exist $\sa$ such that $\nu(Y_\sa)=1$.

\begin{Remark}
Notice that there exists $\underline{s}$ such that $Y_{\underline{s}}=\emptyset$. Indeed, fix $k_0\geq 1$ and let
$$
s_{k_0}:=b_{k_0}-2,\ s_k:=b_k-1 \text{ for }k\neq k_0.
$$
Suppose that $Y_{\underline{s}}\neq\emptyset$ and take $x\in Y_{\underline{s}}$. Then there exist $n,m\in\text{supp}(x)$ such that
$n-m\not\equiv 0\bmod b_{k_0}$. This is however impossible since $n-m\equiv 0 \bmod b_k$ for $k\geq 0$, i.e.\ $n=m$.
\end{Remark}
From now on, we will assume that $Y_{\underline{s},\underline{a}}\neq \emptyset$. Recall the following result.
\begin{Prop}[see~\cite{Pec}, discussion before Lemma 3.3]\label{entropY}
We have
$$
h_{top}(S,\overline{Y}_{\underline{s},\underline{a}})=\log 2\cdot\prod_{k\geq 1}\left(1-\frac{s_k}{b_k} \right).
$$
\end{Prop}
Let $x\in \overline{Y}_{\underline{s},\underline{a}}$, fix $K\geq 1$, $n\in\Z$ and consider $x[n,n+b_1\cdot\ldots\cdot b_K-1]$. Since, by Chinese Remainder Theorem, the map
$$
W\colon \{n,n+1,\dots,n+b_1\cdot\ldots\cdot b_K-1\} \to \prod_{k=1}^{K} \Z/ b_k\Z
$$
given by $W(m)=(m \bmod b_1,\dots, m\bmod b_K)$ is a bijection, therefore
\begin{align*}
|\text{supp}(x) \cap &\{n,n+1,\ldots,n+b_1\cdot\ldots\cdot b_K-1\}|\\
&=|W(\text{supp}(x) \cap \{n,n+1,\ldots,n+b_1\cdot\ldots\cdot b_K-1\})|\\
&\leq \prod_{k=1}^{K}|\text{supp}(x) \cap \{n,n+1,\ldots,n+b_1\cdot\ldots\cdot b_K-1\} \bmod b_k|\\
&\leq \prod_{k=1}^{K} |\text{supp}(x) \bmod b_k|\leq \prod_{k=1}^{K}(b_k-s_k).
\end{align*}
It follows that
\begin{equation}\label{29.05.a}
\frac{|\text{supp}(x) \cap \{n,n+1,\dots,n+b_1\cdot\ldots\cdot b_K-1\}|}{b_1\cdot\ldots\cdot b_K}\leq \prod_{k=1}^K (1-\frac{s_k}{b_k}).
\end{equation}
We will also need the following simple lemma.
\begin{Lemma}\label{lm:7}
Let $X\subset \{0,1\}^\Z$ be closed and shift invariant, and let $\widetilde{X}\subset \{0,1\}^\Z$ be the smallest hereditary system containing $X$. Suppose additionally that for some $d,d'\geq 0$, for any $\vep>0$ there exists $n_0\in\N$ such that for all $n\geq n_0$ and all $B\in \{0,1\}^n$ which occur on $X$
$$
\frac{|\{i: B(i)=1\}|}{n}\in (d-\vep,d'+\vep).
$$
Then $d\log 2\leq h_{top}(S,\widetilde{X})\leq h_{top}(S,X)+d' \log 2$.
\end{Lemma}
\begin{proof}
Let $\vep>0$ and let $n_0\in\N$ be as in the assumptions of the lemma. Given $n\geq1$, denote by $p_n(\widetilde{X})$ and $p_n(X)$ the number of  $n$-blocks occurring on $X$ and $X$, respectively. Notice that the following procedure yields all $n$-blocks occurring on $\widetilde{X}$:
\begin{enumerate}[(i)]
\item pick an $n$-block occurring on $X$,
\item replace some of the $1$s with $0$s.
\end{enumerate}
Therefore,
$$
p_n(\widetilde{X})\leq p_n(X)\cdot 2^{n(d'+\vep)}
$$
for $n\geq n_0$. On the other hand, by fixing one particular $n$-block occurring on $X$ and exhausting all possibilities given by (ii) of the above procedure, we obtain
$$
2^{n(d-\vep)}\leq p_n(\widetilde{X}) \text{ for }n\geq n_0.
$$
This implies
$$
(d-\vep)\log 2\leq h_{top}(S,\widetilde{X})\leq h_{top}(S,X)+(d'+\vep)\log 2
$$
and the result follows.
\end{proof}
As an immediate consequence Proposition~\ref{entropY},~\eqref{29.05.a} and  Lemma~\ref{lm:7} (with $d=d'=0$), we obtain:
\begin{Cor}
If  $h_{top}(S,\overline{Y}_{\underline{s},\underline{a}})=0$ then also the hereditary subshift determined by $\overline{Y}_{\underline{s},\underline{a}}$ is of zero topological entropy.
\end{Cor}
Recall (see~\cite{Kw}) that hereditary subshifts of zero topological entropy are uniquely ergodic with $\delta_{(\ldots,0,0,0,\ldots)}$ being the only invariant measure.
Thus, we have shown the following:
\begin{Cor}
If  $h_{top}(S,\overline{Y}_{\underline{s},\underline{a}})=0$ then $\mathcal{P}(S,\overline{Y}_{\underline{s},\underline{a}})=\{\delta_{(\ldots,0,0,0,\ldots)}\}$. In particular, $\mathcal{P}(S,{Y}_{\underline{s},\underline{a}})=\emptyset$.
\end{Cor}
\begin{Remark}
Notice that it is possible that $h_{top}(S,\overline{Y}_{\underline{s},\underline{a}})=0$ and ${Y}_{\underline{s},\underline{a}}\neq\emptyset$. Indeed, if $s_k=b_k-1$ for all $k\geq 1$ then
\begin{align*}
{Y}_{\underline{s},\underline{a}}&=\{S^n (\ldots,0,0,1,0,0,\ldots) : n\in\Z\},\\
\overline{Y}_{\underline{s},\underline{a}}&=\{S^n (\ldots,0,0,1,0,0,\ldots) : n\in\Z\}\cup \{(\ldots,0,0,0,\ldots)\}.
\end{align*}
\end{Remark}
Fix $\underline{s},\underline{a}$ and suppose that $h_{top}(\overline{Y}_{\underline{s},\underline{a}})>0$.  Let
$$
\Omega_{\underline{s},\underline{a}}=\prod_{k\geq 1}\Z/b_k'\Z, \text{ where }b_k'\text{ are as in}~\eqref{primy}.
$$
We define $\varphi_{\underline{s},\underline{a}}\colon \Omega_{\underline{s},\underline{a}}\to \{0,1\}^\Z$ by
$$
\varphi_{\underline{s},\underline{a}}(\omega)(n)=\begin{cases}
1&\text{if } \omega(k)-a_i^k+n\neq 0 \bmod b_k \text{ for all }k\geq 1, 1\leq i\leq s_k,\\
0&\text{otherwise}.
\end{cases}
$$
Fix $k\geq 1$, $z\in\Z/b_k\Z$ and let
$$
E^{\underline{s},\underline{a}}_{k,z}:=\{\omega\in\Omega_{\underline{s},\underline{a} }: \varphi_{\underline{s},\underline{a}}(\omega)(-z+sb_k')=0\text{ for all }s\geq 1\}.
$$
Next, we define
$$
(\Omega_{\underline{s},\underline{a}})_0':=\Omega_{\underline{s},\underline{a}}\setminus \bigcup_{k\geq 1}\bigcup_{z\in \Z/b_k'\Z}\left( \left( \bigcap_{i=1}^{s_k} E^{\underline{s},\underline{a}}_{k,z-a_i^k}\right) \setminus\left\{\omega\in \Omega_{\underline{s},\underline{a}} : \omega(k)=z \right\}\right)
$$
and we put
$$
(\Omega_{\underline{s},\underline{a}})_0:=\bigcap_{k\in\Z} T^k (\Omega_{\underline{s},\underline{a}})_0'
$$
We claim that $\varphi_{\underline{s},\underline{a}}$ is 1-1 on $(\Omega_{\underline{s},\underline{s}})_0$ (cf.\ Lemma~\ref{luz1}) Moreover,
\begin{equation}\label{omegazero}
\mathbb{P}_{\underline{s},\underline{a}}\left( \left( \bigcap_{i=1}^{s_k} E^{\underline{s},\underline{a}}_{k,z-a_i^k}\right) \setminus\left\{\omega\in \Omega_{\underline{s},\underline{a}} : \omega(k)=z \right\}\right)=0,
\end{equation}
where $\mathbb{P}_{\underline{s},\underline{a}}$ is the normalized Haar measure on $\Omega_{\underline{s},\underline{a}}$. This shows that
$$
\mathbb{P}_{\underline{s},\underline{a}}((\Omega_{\underline{s},\underline{a}})_0)=1.
$$
The proof of~\eqref{omegazero} is essentially the same as the proof of Proposition 3.2 in~\cite{Ab-Le-Ru}. One of the important steps in this proof is to show that
\begin{multline}\label{26.5.d}
\mathbb{P}_{\underline{s},\underline{a}}(\{\omega\in\Omega_{\underline{s},\underline{a}}: \varphi_{\underline{s},\underline{a}}(\omega)(n)=0\})\\
=\mathbb{P}_{\underline{s},\underline{a}}\left(\bigcap_{k\geq 1}\left\{\omega\in\Omega_{\underline{s},\underline{a}}: w(k)-a_i^k+n\neq 0\bmod b_k \text{ for all }1\leq i\leq s_k \right\}\right)
\end{multline}
is strictly positive. To see that this is indeed the case, notice first that for any set $A\subset \Z/b_k\Z$ such that $A+b'_k=A \bmod b_k'$, we have
\begin{multline}\label{26.5.b}
\{\omega\in \Z/b_k'\Z : \omega(k)\neq a \bmod b_k \text{ for all }a\in A\}\\
=\{\omega\in \Z/b_k'\Z : \omega(k)\neq a \bmod b_k' \text{ for all }a\in A\cap \Z/b_k'\Z\}.
\end{multline}
Moreover, since $A=\bigcup_{j=0}^{b_k/b_k'-1}A\cap\{\{0,\ldots,b_k'-1\}+j\}$ and
$|A\cap\{\{0,\ldots,b_k'-1\}+j\}|$ does not depend on $j$, we obtain $|A|=b_k' \cdot |A\cap \Z/b_k'\Z|$. Applying this to $A=\{a_i^k-n : 1\leq i\leq s_k\}$ we conclude that
\begin{equation}\label{26.5.c}
|\{a_i^k-n:1\leq i\leq s_k\}\cap \Z/b_k'\Z|=\frac{s_k\cdot b_k'}{b_k}.
\end{equation}
Using~\eqref{26.5.b} and~\eqref{26.5.c}, we obtain
$$
|\left\{\omega\in\Z/b_k'\Z: w(k)-a_i^k+n\neq 0\bmod b_k \text{ for all }1\leq i\leq s_k \right\}|=b_k'-\frac{s_k\cdot b_k'}{b_k}.
$$
This, in view of~\eqref{26.5.d}, gives indeed
$$
\mathbb{P}_{\underline{s},\underline{a}}(\{\omega\in\Omega_{\underline{s},\underline{a}}: \varphi_{\underline{s},\underline{a}}(\omega)(n)=0\})=\prod_{k\geq 1} \Big(1-\frac{s_k}{b_k}\Big)=h_{top}(\overline{Y}_{\underline{s},\underline{a}})>0.
$$
\begin{Remark}\label{274}
Notice that the above calculation shows in particular that
$$
\prod_{k\geq 1}\left(1-\frac{1}{b_k'} \right)\geq \prod_{k\geq 1}\left(1-\frac{s_k\cdot\frac{b_k'}{b_k}}{b_k'} \right)>0,
$$
i.e.\ $\{b_k' :k\geq 1\}$ yields a free system.
\end{Remark}
We also define $\theta_{\underline{s},\underline{a}}\colon Y_\sa \to \Omega_\sa$  in the following way:
$$
\theta_\sa(y)=\omega \iff -\omega(k)+a_i^k \not\in\text{supp}(y) \bmod b_k \text{ for all }1\leq i\leq s_k.
$$
Moreover, denote by $T_\sa\colon\Omega_\sa\to\Omega_\sa$ the map given by
$$
T_\sa\omega=\omega+(1,1,\ldots)=(\omega(1)+1,\omega(2)+1,\ldots),
$$
where $\omega=(\omega(1),\omega(2),\ldots)$.
\begin{Lemma}[cf.\ Lemma~\ref{l4}] \label{275}
We have:
\begin{enumerate}[(i)]
\item\label{F1}
$\theta_\sa$ is equivariant, i.e.\ $T_\sa\circ \theta_\sa=\theta_\sa \circ S$.
\item\label{F2}
For each $\omega\in \Omega_\sa$ and $y\in Y_\sa$ such that $\theta(y)=\omega$, we have $y\leq \va_\sa(\omega)$.
\item\label{F3}
$\va_\sa((\Omega_\sa)_0)\subset Y_\sa$ (in particular, $\theta_\sa\circ \va_\sa|_{\Omega_0}=id_{\Omega_0}$).
\end{enumerate}
\end{Lemma}

For $n\in\N$ let $M^{(n)}\colon (\{0,1\}^\Z)^{\times n}\to \{0,1\}^\Z$ be given by
$$
M^{(n)}((x^{(1)}_i)_{i\in\Z},\dots,(x^{(n)}_i)_{i\in\Z})=(x^{(1)}_i\cdot \ldots \cdot x^{(n)}_i)_{i\in\Z}.
$$
Moreover, we define $M^\infty \colon (\{0,1\}^\Z)^{\N}\to \{0,1\}^\Z$ as
$$
M^{(\infty)}((x^{(1)}_i)_{i\in\Z},(x^{(2)}_i)_{i\in\Z},\dots)=(x^{(1)}_i\cdot x^{(2)}_i\cdot \ldots  )_{i\in\Z}.
$$
\begin{Lemma}\label{may21a}
We have
$
(\varphi_\sa)_\ast (\PP_\sa)=M^{(\infty)}_\ast (\rho),
$
where $\rho$ is a joining of a countable number of copies of $(S,\{0,1\}^\Z,\nu_{\mathscr{B}'})$.
\end{Lemma}
\begin{proof}
For $i\geq 1$ define $R^{(i)}\colon \Omega_\sa\to\Omega_\sa$ by $R^{(i)}(\omega)=(\omega(k)-\widetilde{a}_i^k)_{k\geq 1}$, where
$$
\widetilde{a}_i^k=\begin{cases}
a_i^k&\text{if } 1\leq i\leq s_k,\\
a_{s_k}^k&\text{if } i >s_k.
\end{cases}
$$
It follows from~\eqref{26.5.b} applied to $A=\{a_i^k-n : 1\leq i\leq s_k\}$ that
$$
\varphi_\sa(\omega)=M^{(\infty)}(\varphi'\circ R^{(1)}(\omega),\varphi'\circ R^{(2)}(\omega),\ldots),
$$
where $\varphi'\colon \Omega_\sa\to\{0,1\}^\Z$ is given by~\eqref{e1} with $\mathscr{B}$ replaced with $\mathscr{B}'$. Thus,
$$
(\varphi_\sa)_\ast(\PP_\sa)=M^{(\infty)}_\ast \circ (\varphi'\circ R^{(1)}\times \varphi'\circ R^{(2)}\times \dots)_\ast (\PP_\sa).
$$
Since $R^{(i)}_\ast(\PP_\sa)=\PP_\sa$ for each $i\geq 1$, it follows that
$$
\rho:=(\varphi'\circ R^{(1)}\times \varphi'\circ R^{(2)}\times \dots)_\ast (\PP_\sa)
$$
is indeed a joining of a countable number of copies of $\varphi'_\ast(\PP_\sa)=\nu_{\mathscr{B}'}$.
\end{proof}

\begin{Lemma}\label{may21b}
Let $\nu_1,\dots, \nu_n,\nu_{n+1}\in\mathcal{P}(S,\{0,1\}^\Z)$. Then for any joinings
\begin{itemize}
\item
$\rho_{1,n}\in J((S,\{0,1\}^\Z,\nu_1),\dots, (S,\{0,1\}^\Z,\nu_n))$,
\item
$\rho_{(1,n),n+1}\in J((S,\{0,1\}^\Z,M^{(n)}_\ast(\rho_{1,n})),(S,\{0,1\}^\Z,\nu_{n+1}))$
\end{itemize}
there exist:
\begin{itemize}
\item
${\rho}_{2,n+1}\in J((S,\{0,1\}^\Z,\nu_2),\dots,(S,\{0,1\}^\Z,\nu_n),(S,\{0,1\}^\Z,n+1))$,
\item
${\rho}_{1,(2,n+1)}\in J((S,\{0,1\}^\Z,\nu_1),(S,\{0,1\}^\Z,M^{(n)}_\ast({\rho}_{2,n+1})))$
\end{itemize}
such that $M_\ast(\rho_{(1,n),n+1})=M_\ast(\rho_{1,(2,n+1)})$.\footnote{We could write this property as $M_\ast(M^{(n)}_\ast(\nu_1\vee\dots\vee\nu_n)\vee \nu_{n+1})=M_\ast(\nu_1\vee M^{(n)}_\ast(\nu_2\vee \dots\vee\nu_n\vee\nu_{n+1}))$. However, until we say which joining we mean by each symbol $\vee$, this expression has no concrete meaning.}
\end{Lemma}
\begin{proof}
Clearly, $(S,\{0,1\}^\Z,M^{(n)}_\ast(\rho_{1,n}))$ is a factor of $(S^{\times n},(\{0,1\}^\Z)^{\times n},\rho_{1,n})$. Let $\widehat{\rho}_{(1,n),n+1}$ be the relatively independent extension of $\rho_{(1,n),n+1}$ to a joining of $(S^{\times n},(\{0,1\}^\Z)^{\times n},\rho_{1,n})$ and $(S,\{0,1\}^\Z,\nu_{n+1})$. Then
$$
(M^{(n)}\times Id)_\ast(\widehat{\rho}_{(1,n),n+1})=\rho_{(1,n),n+1}
$$
and
$$
\widehat{\rho}_{(1,n),n+1}\in J((S,\{0,1\}^\Z,\nu_1),(S^{\times n},(\{0,1\}^\Z)^{\times n},\rho_{2,n+1})),
$$
where ${\rho}_{2,n+1}$ is a projection of $\widehat{\rho}_{(1,n),n+1}$ onto the last $n$ coordinates. Let
$$
\rho_{1,(2,n+1)}:=(Id\times M^{(n)})_{\ast}(\widehat{\rho}_{(1,n),n+1}).
$$
Clearly, $\rho_{1,(2,n+1)}\in J((S,\{0,1\}^\Z,\nu_1),(S,\{0,1\}^\Z,M^{(n)}_\ast({\rho}_{2,n+1})))$. Moreover,
\begin{multline*}
M_\ast(\rho_{1,(2,n+1)})=M_\ast\circ (Id\times M^{(n)})_\ast (\widehat{\rho}_{(1,n),n+1})\\
=M_\ast\circ (M^{(n)}\times Id)_\ast (\widehat{\rho}_{(1,n),n+1})=M_\ast (\rho_{(1,n),n+1})
\end{multline*}
and the assertion follows.
\end{proof}
\begin{Remark}\label{may21c}
The above lemma remains true when we consider infinite joinings, i.e.\ instead of $\nu_1,\dots,\nu_n$ we have $\nu_1,\nu_2,\dots$, and instead of $M^{(n)}$ we consider $M^{(\infty)}$.
\end{Remark}

\begin{proof}[Proof of Theorem~\ref{twYall}]
Fix $\nu\in\mathcal{P}^e(S,X_\eta)$ and let $\sa$ be such that $\nu(Y_\sa)=1$. In view of Lemma~\ref{may21a},~Lemma~\ref{may21b} and Remark~\ref{may21c}, what we need to show is that there exists $\widetilde{\rho}\in\mathcal{P}(S\times S,\{0,1\}^\Z\times\{0,1\}^\Z)$ such that
the projection of $\widetilde{\rho}$ onto the first coordinate equals $(\varphi_\sa)_\ast(\PP_\sa)$ and $M_\ast(\widetilde{\rho})=\nu$.

The remaining part of the proof goes exactly along the same lines as the proof of Theorem~\ref{twY}, with the following modification: we need to replace some objects related to $Y$ by their counterparts related to $Y_\sa$. Namely, instead of $\Omega,\ \Theta,\ Y_\infty,\ \widetilde{T},\ \Omega_0,\ \Phi,\ Y_0 \text{ and }\Psi$,
we use
$$
\Omega_\sa,\ \Theta_\sa,\ (Y_\sa)_\infty,\  \widetilde{T}_\sa,\ (\Omega_\sa)_0,\ \Phi_\sa,\ (Y_\sa)_0\text{ and } \Psi_\sa,
$$
where
\begin{itemize}
\item
$\Theta_\sa\colon Y_\sa\to \Omega_\sa\times\{0,1\}^\Z$ is given by $\Theta_\sa(y):=(\theta_\sa(y),\widehat{y}_{\varphi_\sa(\theta_\sa y)})$,
\item
$(Y_\sa)_\infty=\{y\in {Y_\sa} : |\text{supp\ }\varphi_\sa(\theta_\sa(y)) \cap (-\infty,0)|= |\text{supp\ }\varphi_\sa(\theta_\sa(y)) \cap (0,\infty)|=\infty\}$,
\item
$\widetilde{T}_\sa\colon \Omega_\sa\times \{0,1\}^\Z\to \Omega_\sa\times \{0,1\}^\Z$ given by
$$
\widetilde{T}_\sa(\omega,x)=\begin{cases}
(T_\sa\omega, x)&\text{if }\varphi_\sa(\omega)(0)=0\\
(T_\sa\omega, Sx)&\text{if } \varphi_\sa(\omega)(0)=1,
\end{cases}
$$
\item
$\Phi_\sa(\omega, x)$ is the unique element in $X_\eta$ such that
\begin{enumerate}[(i)]
\item\label{eq:C1}
$\Phi_\sa(\omega,x)\leq \varphi_\sa(\omega)$,
\item
$({\Phi_\sa(\omega,x)})^{\hat{}}_{\varphi_\sa(\omega)}=x$, i.e.\ the consecutive coordinates of $x$ can be found in $\Phi_\sa(\omega,x)$ along $\varphi_\sa(\omega)$,
\end{enumerate}
\item
$(Y_\sa)_0=\theta_\sa^{-1}((\Omega_\sa)_0)$,
\item
$\Psi_\sa(\omega,x)=(\omega, \widehat{x}_{\varphi_\sa(\omega)})$.
\end{itemize}
\end{proof}

We may also extend Theorem~\ref{benji} in the following way:
\begin{Th}
Each of the subshifts $\overline{Y}_\sa$ is intrinsically ergodic.
\end{Th}
The proof is very similar to the one of Theorem~\ref{benji} presented at the end of the Section~\ref{joa}. The only difference is that instead of $\Omega,\va,\widetilde{T},C$ we use $\Omega_\sa, \va_\sa,\widetilde{T}_\sa,C_\sa$, where $C_\sa:=\{\omega\in\Omega_\sa : \varphi_\sa(\omega)(0)=1\}$.

Moreover, using Remark~\ref{274} and Lemma~\ref{275}, we obtain the following:
\begin{Cor}\label{restmeasure}For each  $\nu\in \cp^e(S,X_\eta)$ the discrete rational part of the spectrum of the corresponding dynamical system $(S,X_\eta,\nu)$ contains all $b'_1\cdot\ldots\cdot b'_k$-roots of unity, $k\geq1$, where $\mathscr{B}'=\{b_k':k\geq 1\}$ is such that~\eqref{in1} and~\eqref{in2} are satisfied.
\end{Cor}

\begin{Cor}\label{nomeasure}
Let $1<b_k'|b_k$ for $k\geq 1$. The following are equivalent:
\begin{enumerate}[(a)]
\item\label{GA}
there exists a measure $\nu\in\mathcal{P}^{e}(S,X_\eta)$ such that the rational discrete spectrum of $(S,X_\eta,\nu)$ is equal to all $b'_1\cdot\ldots\cdot b'_k$-roots of unity
\item\label{GB}
$\sum_{k\geq 1}1/b_k'<+\infty$.
\end{enumerate}
In particular, no ergodic measure for the square-free subshift yields the dynamical system whose spectrum consists of all $p_1\cdot\ldots\cdot p_k$-roots of unity, $k\geq 1$.
\end{Cor}
\begin{proof}
To see that~\eqref{GB} implies~\eqref{GA}, it suffices to take $\nu=\nu_{\mathscr{B}'}$. Suppose now that $\nu\in\mathcal{P}^e(S,X_\eta)$ satisfies~\eqref{GA}. It follows by Corollary~\ref{restmeasure} that there exists $1<b_k''|b_k$, $k\geq 1$ such that $\sum_{k\geq 1}1/b_k''<+\infty$ and the discrete part of the spectrum of $(S,X_\eta,\nu)$ contains all $b''_1\cdot\ldots\cdot b''_k$-roots of unity, $k\geq1$. In particular, it contains all $b_k''$-roots of unity, $k\geq 1$. Therefore, for each $k\geq 1$ there exists $\ell_k$ such that $b_k''|b_1'\cdot\ldots\cdot b_\ell'$. Using~\eqref{in2}, we obtain immediately that $\ell\geq k$ and $b_k'' | b_k'$, which yields~\eqref{GB}.
\end{proof}

\subsection{Combinatorics}
\begin{Prop}\label{appl1}
Assume that $\mathscr{B}=\{b_k: k\geq1\}$ and $\mathscr{B}'=\{b'_k: k\geq1\}$ satisfy~\eqref{f1}.
If $X_{\mathscr{B}}=X_{\mathscr{B}'}$ then $\mathscr{B}=\mathscr{B}'$.
\end{Prop}
\begin{proof}We can additionally assume that $b_1<b_2<\ldots$ and also $b'_1<b'_2<\ldots$ Denote by $(T,\Omega,\PP)$ and $(T',\Omega',\PP')$ the corresponding odometers and by $\va\colon\Omega_0\to X_{\mathscr{B}}$, $\va'\colon\Omega'_0\to X_{\mathscr{B}'}$ the relevant genuine embeddings, see~\eqref{e1} and Lemma~\ref{luz1}.

We claim now that $\nu_{\mathscr{B}'}(Y)=1$. Indeed, notice first that $\nu_{\mathscr{B}}\ast B(\nicefrac12,\nicefrac12)=\nu_{\mathscr{B}'}\ast B(\nicefrac12,\nicefrac12)$ since both measures are of maximal entropy on $X_\eta=X_{\eta'}$ and $(S,X_\eta)$ is intrinsically ergodic. Suppose that $\nu_{\mathscr{B}'}(Y)=0$. Then
$\nu_{\mathscr{B}'}(\bigcap_{k\geq1}Y_{k,s_k})=1$ with at least one $s_k\geq 2$, see Remark~\ref{uwaga9}. But the set
$$\bigcap_{k\geq1}\bigcup_{r_k\geq s_k}Y_{k,r_k}$$
is hereditary and clearly $\nu_{\mathscr{B}'}\ast B(\nicefrac12,\nicefrac12)$ is concentrated on it. On the other hand, by Lemma~\ref{l3}, the measure of maximal entropy is concentrated on $Y$, a contradiction and our claim follows. Since $\theta_\ast(\nu_{\mathscr{B}'})=\PP$, we also have $\nu_{\mathscr{B}'}(Y_0)=1$. In other words, we may assume without loss of generality that $\varphi'(\Omega')\subset Y_0=\varphi(\Omega_0)$.

Now, for $\omega'\in\Omega'_0$, there exists $\omega\in \Omega_0$ such that $\va'(\omega')\leq \va(\omega)$ (in fact, $\omega=\theta(\va'(\omega'))$, see Lemma~\ref{l4}~\eqref{F2}). Fixing now~$\omega$ and reversing the roles, we find $\omega^{\prime\prime\prime}\in\Omega'_0$ such that $\va(\omega)\leq\va'(\omega^{\prime\prime\prime})$. Thus,
$$
\va'(\omega^{\prime})\leq\va'(\omega^{\prime\prime\prime})$$
which,  by~\eqref{mz0} used for $\va'$,  implies that $\omega^{\prime\prime\prime}=\omega^{\prime}$. It follows that $\va(\Omega_0)=\va'(\Omega'_0)$. Now, $\theta_{\ast}(\nu_{\mathscr{B}})=\PP=\theta_{\ast}(\nu_{\mathscr{B}'})$
and $\theta|_{\va(\Omega_0)}$ is 1-1. It follows that \beq\label{mz}\nu_{\mathscr{B}}=\nu_{\mathscr{B}'}.\eeq

Furthermore,
\beq\label{joa1}X_{\mathscr{B}}=X_{\mathscr{B}'}\;\Rightarrow\;b_1=b'_1.\eeq
Indeed, suppose $b_1<b'_1$. Then to obtain~\eqref{joa1}, it is enough to notice that the block $C^1_{\{1,\ldots, b'_1-1\}}\cap C^0_{\{b'_k\}}$ is $\mathscr{B}'$-admissible, while clearly it is not $\mathscr{B}$-admissible.

Set $\widetilde{\mathscr{B}}=\mathscr{B}\setminus\{b_1\}$, $\widetilde{\mathscr{B}}'=\mathscr{B}'\setminus\{b'_1\}$. In view of~\eqref{e2}, for any finite subset $A\subset \N$, we have
$$
\nu_{\mathscr{B}}(C^1_A)=\left(1-\frac{|A\;{\rm mod}\; b_1|}{b_1}\right)\nu_{\widetilde{\mathscr{B}}}(C^1_A)$$
with an analogous formula for $\nu_{\mathscr{B}'}$. In view of~\eqref{joa1} and~\eqref{mz}, we deduce $\nu_{\widetilde{\mathscr{B}}}=
\nu_{\widetilde{\mathscr{B}'}}$ whence $X_{\widetilde{\mathscr{B}}}=X_{\widetilde{\mathscr{B}'}}$ (the Mirsky measure has full topological support). Using again~\eqref{joa1}, we obtain $b_2=b'_2$, and by continuing, we conclude $\mathscr{B}=\mathscr{B}'$.
\end{proof}
\begin{Remark}
Given a subset $A\subset\N$ denote
$$\widetilde{A}:=\{C\subset\Z: (\forall C\supset E,\;E\; {\rm is \ finite})(\exists k\in\Z)\;\; E+k\subset A\}.$$
The result obtained in Proposition~\ref{appl1} can be reformulated as follows. Assume that $\mathscr{B}=\{b_k: k\geq1\}$ and $\mathscr{B}'=\{b'_k: k\geq1\}$ satisfy~\eqref{f1} and let $F_{\mathscr{B}}$, $F_{\mathscr{B}'}$ stand for the sets of $\mathscr{B}$- and $\mathscr{B}'$-free numbers, respectively. Then
\beq\label{applmaj}
\widetilde{F}_{\mathscr{B}}=\widetilde{F}_{\mathscr{B}'}\;\;\mbox{if and only if}\;\;\mathscr{B}=\mathscr{B}'.\eeq
\end{Remark}
The proof of Proposition~\ref{appl1}, although short, uses however some non-trivial facts, like intrinsic ergodicity of $\mathscr{B}$-free systems. We will now present an elementary proof, due to Stanisław Kasjan, which has an advantage that it also gives a sufficient and necessary condition for $X_\eta\subset X_{\eta'}$.

Let $\mathscr{B}=\{b_k:k\ge 1\}\subseteq\N$  satisfies~\eqref{f1} and assume that $b_1<b_2<\ldots$

\begin{Lemma}\label{lsk1} Let $c\in\N$ be relatively prime to $b_k$ for any $k\geq1$.  Then, for any natural number $r$  the density of the set $\{s\in\N: sc+r\neq0\;{\rm mod}\; b_k\;\mbox{for each}\;k\geq1\}$ equals
$
\prod_{k\geq 1}\left(1-\frac{1}{b_k}\right)
$.\footnote{
If $\sum_{k\geq 1}{1}/{b_k}=+\infty$ then this density equals 0.}
\end{Lemma}
\begin{proof}
Fix $m\geq1$ and consider the
finite probability space $\Z / (b_1\cdot\ldots\cdot b_m) \Z$,
 that is,  the integers mod $M:=b_1\cdot\ldots\cdot b_m$.
By the Chinese Remainder Theorem  the random variables $X_{b_i}(s)$  which equals~$0$ if~$s$ is divisible by~$b_i$ and $1$ otherwise, are independent. It follows that
the event $A$ such that $n \in A$ if and only if $n$ is not divisible by any
of the $b_i$ has probability $p=\prod_{i=1}^m(1 - 1/b_i)$.
Now, if $c$ is relatively prime to $M$, then the  addition
by $c$ modulo $M$ is transitive (the $jc$ mod $M$ for $j=0,\ldots,M-1$  are all distinct and yield all residues mod $M$), the
time average of the $\raz_A(lc +r )$ equals its space average which is $p$. Moreover, note that the number of elements  in the interval $[1,M]$ which are divisible by some $b_{m+i}$  ($i\geq1$) is no more than $M\cdot\sum_{i\geq 1}\frac1{b_{m+i}}$. This gives the density result along the subsequence $M_k:=b_1\cdot\ldots\cdot b_k$, $k\geq1$ and the general result easily follows.
\end{proof}
Note also that a simple induction on finite products shows that
$$\prod_{k\geq1}\left(1-\frac1{b_k}\right)\geq1-\sum_{k\geq1}\frac1{b_k},
$$
so for each $\vep>0$ there exists $N\geq1$ such that
\beq\label{esk2}
\prod_{k\geq N}\left(1-\frac1{b_k}\right)>1-\vep.\footnote{To obtain  \eqref{esk2} we need only that $\sum_{k\geq1}1/b_k<+\infty$.}\eeq

\begin{Lemma}\label{lsk2} Fix $m\geq1$ and assume that $b'\in\N$ is not divisible by any
 $b_1,\ldots,b_m$.
Then there exists a set $A\subset \N$ containing $b'$ elements which is $\{b_1,\ldots,b_m\}$-admissible  and such that $A$ is not $\{b'\}$-admissible.
\end{Lemma}
\begin{proof}
For $i=1,\ldots,b'$ let $e_i$ denote the product of those numbers from the set $\{b_1,\ldots,b_m\}$ which do not divide $i$. If every $b_j$ divides $i$, we set $e_i=1$. We define
$$A=\{i+e_i b':i=1,\ldots,b'\}.$$
By construction, $A$ is not $\{b'\}$-admissible. We will show that $A$ is $\{b_1,\ldots,b_m\}$-admissible by showing that
$0\notin  A$ mod $b_i$ for each $i=1,\ldots,m$.
Indeed, let $i\in\{1,\ldots,b'\}$, $j\in\{1,\ldots,m\}$. If $b_j$ divides $i$, then it is relatively prime to $e_i$ (by the choice of $e_i$ and the assumption that the numbers $b_i$ are pairwise relatively prime). Thus $b_j$ does not divide $i+e_i b'$. In the other case, when $b_j$ does not divide $i$, then it divides $e_i$ and again
$b_j$ does not divide $i+e_i b'$.
\end{proof}

\begin{Prop} \label{appl2} Assume that $\mathscr{B}=\{b_k:k\ge 1\}$ and $\mathscr{B}'=\{b'_k:k\ge 1\}$ satisfy~\eqref{f1}. Then:
\begin{enumerate}[(a)]
\item\label{Na}
$X_{\mathscr{B}}\subset X_{\mathscr{B}'}$ if and only if  for any $b'\in \mathscr{B}'$ there exists $b\in \mathscr{B}$ such that $b$ divides $b'$.
\item\label{Nb}
If $X_{\mathscr{B}}=X_{\mathscr{B}'}$ if and only if $\mathscr{B}=\mathscr{B}'$.
\end{enumerate}
\end{Prop}
\begin{proof}
\eqref{Na} We only need to show if
 $X_{\mathscr{B}}\subset X_{\mathscr{B}'}$ and $b'\in \mathscr{B}'$ then for some $b\in \mathscr{B}$, $b$ divides $b'$.

 Fix  $b'\in \mathscr{B}'$ and assume that $b'$ is not divisible by any  $b\in \mathscr{B}$. Using~\eqref{esk2}, we select  $m\geq1$ so that
  $$
  \prod_{k> m}\left(1-\frac{1}{b_k}\right)>1-\frac{1}{b'}.
  $$
In view of Lemma~\ref{lsk2}, we can find a set $A\subset\N$, $|A|=b'$, which is $\{b_1,\ldots,b_m\}$-admissible and is not $\{b'\}$-admissible, hence is not $\mathscr{B}'$-admissible. Denote $c=b_1\ldots b_m$. It follows that for each $\ell\in\N$, $A+\ell c$ is $\{b_1,\ldots,b_m\}$-admissible and is not $\mathscr{B}'$-admissible.
To complete the proof, it is enough to show that for some $\ell_0$, $A+\ell_0c$ is $\mathscr{B}$-admissible. For this aim, we will show that for some $\ell_0\geq1$
$$(A+\ell_0c)\cap \bigcup_{k>m}b_k\Z=\emptyset.$$
Indeed, if not then  each $K\geq1$
$$
\frac{1}{K}\big|\{1\leq \ell\le K:
(A+\ell c)\cap(\bigcup_{i>m}b_i\Z)\neq\emptyset\}\big|=1.$$
Since $|A|=b'$, it follows that there exists $a\in A$ such that
$$\limsup_{K\rightarrow \infty}\frac{1}{K}\left|\left\{1\leq \ell\le K:
\ell c+a\in\bigcup_{i>m}b_i\Z\right\}\right|\ge
\frac{1}{b'}$$
and we get a contradiction with  the choice of $m$ and Lemma~\ref{lsk1} (applied to $\{b_k: k>m\}$).

\eqref{Nb} Since the elements of $\mathscr{B}$ (resp.\ $\mathscr{B}'$) are pairwise relatively prime,~\eqref{Nb} follows from~\eqref{Na}.
\end{proof}

\section{Hereditary systems of Sturmian origin}

\subsection{Intrinsic ergodicity}\label{section3}
In this section, we will indicate that our method to prove intrinsic ergodicity for $\mathscr{B}$-free systems can be applied to other hereditary systems. We will detail the case of Sturmian hereditary systems but the method applies to many others.

Consider an
irrational rotation $T\colon\T\to\T$, $Tx=x+\alpha$.  Fix an interval
$J=[a,b)\subset\T$ assuming that the numbers $|J|$ and $\alpha$ are independent over $\Q$. Let
$$
X:=\overline{\{(\raz_J(T^n w))_{n\in\Z} : w\in\T\}}.
$$
This is an almost one-one
extension of the circle and the only points that have two representatives
are the orbits of the endpoints of $J$ (see, e.g.,~\cite{Fo}). We will denote the corresponding factor map from $X$ to $\mathbb{T}$ by $\pi$.
\begin{Remark}\label{may1}
Notice that if $x,x'\in\pi^{-1}(a+s\alpha)$, then they differ only at one place, namely, $x[-s]\neq x'[-s]$. It follows that either $x\leq x'$ or $x'\leq x$. The same reasoning applies to the points from the orbit of $b$.
\end{Remark}
It follows that we have a continuous map $\pi\colon X\to\T$ intertwining the rotation $T$ with the shift on $X$, for which $|\pi^{-1}(w)|=1$, except for countably many points $w\in\T$. This allows us to define ``Haar'' measure $\la$ on $X$ (which is the lift of Haar measure $\la_{\T}$ via the map $\pi:X\to\T$). Moreover, $X$ is uniquely ergodic. Let
$$
X_0:=X\setminus \{\pi^{-1}(a+s\alpha), \pi^{-1}(b+s\alpha):s\in \Z\}
$$
and let $\widetilde{X}$ be the hereditary subshift generated by $X$:
$$
\widetilde{X}:=\{z\in\{0,1\}^{\Z}:(\exists x\in X)\;\;z\leq x\}.$$
In view of Lemma~\ref{lm:7} (with $X$ of zero topological entropy), we have
$$
h_{top}(S,\widetilde{X})=\log2\cdot \la(C^1_0).
$$
Finally, let
\begin{align*}
Y&:=\{y\in\widetilde{X} : \text{ all blocks from }X\text{ occur on }y\},\\
Y_0&:=\{y\in Y : (\exists x_0\in X_0)\;\; y\leq x_0\}.
\end{align*}

\begin{Lemma}\label{lb1}
Fix $y\in Y$ and let $x,x'\in X$ be such that $y\leq x$, $y\leq x'$. Then $\pi(x)=\pi(x')$. In particular, if $y\in Y_0$ then $x=x'$, i.e.\ there is exactly one $x=x(y)\in X$ (in fact, $x\in X_0$) such that $y\leq x$.
\end{Lemma}
\begin{proof} Assume that $y\leq x$, $y\leq x'$ for some $x,x'\in X$. Fix $N\geq1$ and choose any maximal block $C$ of length $N$ that occurs on $X$ (i.e.\ if a block $C'$ of length $N$ occurs on $X$ and $C\leq C'$ then $C=C'$). Now, since $y\leq x$ and $C$ occurs on $y$, we can find $\ell\in\Z$ such that $C=y[\ell,\ell+N-1]\leq x[\ell,\ell+N-1]$ and by maximality, $y[\ell,\ell+N-1]=x[\ell,\ell+N-1]$. Hence
$$
x[\ell,\ell+N-1]= x'[\ell,\ell+N-1].
$$
By uniform continuity of $\pi$, if $N$ is large enough, then $\pi(S^{\lfloor 3\ell/2 \rfloor}x)$ is $\vep$-close to $\pi(S^{\lfloor 3\ell/2 \rfloor}x')$. By equivariance and the fact that $T$ is an isometry, $\pi(x)=\pi(x')$.
Finally, if $y\in Y_0$ then $x=x'$.
\end{proof}

In view of Lemma~\ref{lb1} and Remark~\ref{may1}, we can define a Borel map $\theta\colon Y\to X$ by setting
$$
\theta(y)=x \iff x\in X\text{ is the maximal element which dominates }y.
$$
This map is equivariant. Each block at positions $[-k,k]$ that occur on the fiber $\theta^{-1}(x)$  is smaller than $x[-k,k]$ and each block smaller than $x[-k,k]$ does occur on the fiber (as on $y\in Y$ we can see all blocks occurring on $x$). Let $\nu\in\mathcal{P}^e(S,Y_0)$. Then $\theta_\ast(\nu)=\la$ since $(S,X)$ is uniquely ergodic.  It is not hard to see that a measure which maximizes the entropy is obtained by independently changing 1 to 0 with probability $1/2$, cf.\ the definition of $\mu$ in Section~\ref{outline}.

\begin{Lemma}\label{lb2}Each measure of maximal entropy is supported on $Y$.\end{Lemma}
\begin{proof}Fix a word $w$ that occurs in $X$ and contains a $1$.
We will estimate
the number of dictinct $N$-blocks that can occur in the set
$$
Y(w):=\{y\in\widetilde{X}: (\exists x\in X)\;\; y\leq x\text{ and $w$ does not occur on }y\}.
$$
Note that if $\nu\in \cp^e(S,\widetilde{X})$ then either $\nu(Y(w))=0$ or~$1$ since $Y(w)$ is Borel and invariant under the shift. Hence, if in our estimate we get a bound strictly smaller than the
full entropy of $\widetilde{X}$, this will show that a measure of maximal entropy must be supported on $Y$. Indeed, $Y=\bigcap_w Y(w)^c$.

Define $K$ to be the number of ones that occur in $w$.
Now, we let $N$ be large enough so that:
\begin{enumerate}[(a)]
\item\label{b1}
the exponential number of $N$-words in $\widetilde{X}$ is very close to $h_{top}(\widetilde{X})$;
\item\label{b2}
in every word $u$ of length $N$ in $X$ there are a fixed fraction of disjoint occurrences of $w$, say $f>0$;\footnote{This follows from the fact that $(S,X)$ is uniquely ergodic. We apply (uniformly) the ergodic theorem to the cylinder $C$ corresponding to $w$ with $SC\cup\ldots\cup S^{|w|-1}C$ removed.}
\item\label{b3}
The total number of $N$-words in $X$ is exponentially very small.\footnote{This follows from $h_{top}(S,X)=0$.}
\end{enumerate}
Note that $f$ depends on $w$ and the ``very close'' in~\eqref{b1} and the ``very small'' in~\eqref{b3} are chosen
after we know $f$.

Now, when we calculate the number of blocks in $\widetilde{X}$
that are dominated by a fixed $u_0$ of length $N$ in $X$, we get
$2^{aN}$, where $a$ is the frequency of ones in $u_0$.
Now,  $aN = fKN + (a -fK)N$ (here $fKN$ corresponds to the consecutive disjoint occurrences of $w$ in $u_0$).
However, if $w$ does not occur below $u_0$ then the number of
possibilities is only
$
(2^K-1)^{fN} \cdot 2^{(a-fK)N},$
so that the ratio between them is $R^N$, where $R:=(2^K/(2^K-1))^f>1$.
In view of~\eqref{b3}, the entropy of $Y(w)$ will be exponentially comparable with
$2^{\vep N}\cdot (2^K-1)^{fN}\cdot 2^{a-fK}N$, while by~\eqref{b1}, the entropy of $\widetilde{X}$ must be exponentially comparable with $2^{\vep N}\cdot 2^{aN}$ (as the frequency of ones in each $u_0$ is comparable with $a$), and therefore on $Y(w)$, we have a definite drop in the entropy.
\end{proof}

We have now described a full analogy with the $\mathscr{B}$-free systems. Indeed, $X_\eta$ corresponds to $\widetilde{X}$, $\va(\Omega)$ (a symbolic model of the odometer) corresponds to $X$ (see~\eqref{emb2} in Remark~\ref{embedding}), $Y$ and $Y_0$ play the same roles in both cases, and $\theta$ in the Sturmian case is immediately with values in $X$ (it is simpler than in the $\mathscr{B}$-free case as we do not need the map $\va$ to get a symbolic model of the odometer embedded in $X_\eta$).
Now, by repeating the proof of Theorem~\ref{benji}, we obtain the following result.

\begin{Prop}\label{pSturmian} Let $(S,\widetilde{X})$ be a Sturmian hereditary system described above. Then it is intrinsically ergodic.
\end{Prop}

\subsection{Absence of intrinsic ergodicity}\label{section4}
\subsubsection{Tools}
Given a block $C\in \{0,1\}^n$, let $x_C$ be the infinite concatenation of $C$ and let $X_C:={\mathcal{O}(x_C)}\subset \{0,1\}^\Z$.
Finally, let $\widetilde{X}_C\subset \{0,1\}^\Z$ be the smallest hereditary system containing $X_C$. We may assume without loss of generality that the smallest period of $x_C$ is equal to $|C|$.
It follows directly from Lemma~\ref{lm:7} that
\beq\label{uw:7}
h_{top}(S,\widetilde{X}_C)=d\log 2,\;\;\mbox{where}\;\;d=|{\rm supp}\,C|/|C|.
\eeq
Let $\nu_C$ be the Haar measure on $X_C$ and let
$$
\mu_C^{p}:= \nu_C \ast \kappa, \text{ where }\kappa=B(p,1-p), p\in(0,1).
$$
Then $\mu_C\in \cp^e(S,\widetilde{X}_C)$. Moreover, similar arguments as in Section~\ref{prody} yield  \beq\label{formula}h(\mu_C^p)=-d(p\log p+(1-p)\log (1-p)).\eeq In particular, $h(\mu_C^{1/2})=h_{top}(\widetilde{X}_C)$.\footnote{From now on, we will simplify entropy notation  if no confusion arises.}

\begin{Lemma}\label{lm:tran}
Let $\widetilde{X}\subset \{0,1\}^\mathbb{Z}$ be a hereditary subshift. Then there exists $x\in \{0,1\}^\Z$ such that $\widetilde{X}\subset \overline{\mathcal{O}(x)}$, $\overline{\mathcal{O}(x)}$ is hereditary and $h_{top}(x)=h_{top}(\widetilde{X})$.
\end{Lemma}
\begin{proof}
For $n\in\N$ let $\mathcal{C}_n:=\{B_1^n,\dots,B_{l_n}^n\}$ be the family of all $n$-blocks occurring on $\widetilde{X}$. Let $L:=\lceil \frac{\log 2}{h_{top}(\widetilde{X})} \rceil$ and denote by $Z_n$ the $Ln$-block consisting of $Ln$ zeroes. Let $n_1\in \N$ and define $x[1,m_1]$, where $m_1=l_{n_1}n_1+l_{n_1}Ln_1=l_{n_1}(L+1)n_{1}$ by concatenating:
$$
B_1^{n_1}, Z_{n_1},B_2^{n_1},Z_{n_1},\dots, B_{l_{n_1}}^{n_1}, Z_{n_1}.
$$
Now we begin the inductive procedure. Suppose that $n_1<\dots <n_{k-1}$ are chosen and $m_{k-1}$ is the largest integer such that  $x[1,m_{k-1}]$ is already defined. Let $n_{k}=m_{k-1}$ and let
$$
\mathcal{E}_{k}:=\{E\in\{0,1\}^{n_k}: E\leq x[1,n_k]\}=\{E_1^{k},\dots, E_{s_{k}}^{k}\}.
$$
Define $x[m_{k-1}+1,\dots, m_{k}]$, where
$$
m_{k}=m_{k-1}+(l_{n_{k}}+s_k)(L+1)n_{k}
$$
by concatenating:
\begin{equation}\label{eq:A}
Z_{n_{k}}, B_1^{n_{k}}, Z_{n_{k}}, B_2^{n_{k}}, \dots, Z_{n_{k}}, B_{l_{n_{k}}}^{n_{k}},Z_{n_{k}}, E_1^{k}, Z_{n_{k}}, E_2^{k}, \dots, Z_{n_{k}}, E_{s_{k}}^{k}.
\end{equation}
Continuing this procedure,  we obtain $x\in \{0,1\}^\N$.

Notice that $x$ has the following properties:
\begin{itemize}
\item
for any $B\in x$ and any $C\leq B$ we have $C\in x$, i.e.\ the orbit closure of $x$ yields a hereditary shift,
\item
for all $B\in \widetilde{X}$ we have $B\in x$, hence $X\subset \overline{\mathcal{O}(x)}$ and $h_{top}(x)\geq h_{top}(\widetilde{X})$.
\end{itemize}
We will estimate now from above the number of $n_k$-blocks occurring on $x$. Notice that in $x[1,n_k]$ any two consecutive blocks $C,C'\in \mathcal{C}_{n_{k-1}}\cup \mathcal{E}_{k-1}$ are separated by $Z_{n_{k-1}}$ (cf.~\eqref{eq:A} with $k-1$ instead of $k$). Therefore,
$$
d_k:=\frac{|\{1\leq i\leq n_k : x(i)=1\}|}{n_k}
\leq \frac{n_{k-1}}{(L+1)n_{k-1}}=\frac{1}{\lceil\frac{\log 2}{h_{top}(\widetilde{X})} \rceil+1}\leq \frac{h_{top}(\widetilde{X})}{\log 2}.
$$
This implies that
\begin{equation}\label{eq:A1}
|\mathcal{E}_k|=2^{d_kn_k}\leq 2^{\frac{h_{top}(\widetilde{X})}{\log 2}n_k}.
\end{equation}
Moreover, notice that any $n_k$-block $B$ occurring on $x$ satisfies (at least) one of the following:
\begin{itemize}
\item $B=B'Z$, where $Z$ is a (possibly empty) block consisting of zeroes and for some $B''$ we have $B''B'\in \mathcal{C}_{n_k}\cup \mathcal{E}_{k}$,
\item $B=ZB'$, where $Z$ is a (possibly empty) block consisting of zeroes and for some $B''$ we have $B'B''\in  \mathcal{C}_{n_k}\cup \mathcal{E}_{k}$.
\end{itemize}
It follows from~\eqref{eq:A1} that the number of such blocks with $B''B'\in \mathcal{E}_{k}$ or $B'B''\in\mathcal{E}_{k}$ is bounded from above by
$$
p_{n_k}^\mathcal{E}:=(2n_k+1)2^{\frac{h_{top}(\widetilde{X})}{\log 2}n_k}.
$$
Moreover, the number of such blocks with $B''B'\in\mathcal{C}_{n_k}$ or $B'B''\in\mathcal{C}_{n_k}$ is bounded from above by
$$
p_{n_k}^{\mathcal{C}}:=(2n_k+1)p_{n_k}(\widetilde{X}),
$$
where $p_{n_k}(\widetilde{X})$ stands for the number of $n_k$-blocks occurring on $\widetilde{X}$. Therefore
$$
h_{top}(x)\leq \max\left\{\lim_{k\to \infty} \frac{1}{n_k}\log p_{n_k}^\mathcal{E},\lim_{k\to \infty}\frac{1}{n_k}\log p_{n_k}^{\mathcal{C}}\right\}=h_{top}(\widetilde{X})
$$
and the result follows.
\end{proof}

\subsubsection{More than one measure of maximal entropy}
For $A:=101001000$, $B:=101000100$ consider $(\widetilde{X}_A,\mu_A)$ and $(\widetilde{X}_B,\mu_B)$, with $\mu_A=\mu_A^{1/2}$, $\mu_B=\mu_B^{1/2}$. Let $\widetilde{X}:=\widetilde{X}_A\cup \widetilde{X}_B$.
\begin{Prop}\label{pr:1}
The measures $\mu_A$ and $\mu_B$ are ergodic and such that $h_{top}(\widetilde{X})=h(\mu_A)=h(\mu_B)=\frac{1}{3}\log 2$. Moreover, $\mu_A\neq \mu_B$.
\end{Prop}
\begin{proof}
The first part follows easily from~\eqref{uw:7} and~\eqref{formula}. For the second part of the assertion, let
$
Y_A:=\{x\in \widetilde{X} : |\{i : x[i,\dots,i+8]=A\}|=\infty\}.
$
To conclude, it suffices to notice that $\mu_A(Y_A)=1$ (cf.\ Lemma~\ref{lb2}), whereas $\mu_B(Y_A)=0$ since for no $i$, $A\leq (BB)[i,i+8]$.
\end{proof}
Thus we obtain the following corollary which gives the answer to a question raised in~\cite{Kw}.
\begin{Cor}\label{maj2}
There exists a hereditary shift with more than one ergodic measure of maximal entropy.
\end{Cor}
Moreover, as an immediate consequence of Corollary~\ref{maj2} and of Lemma~\ref{lm:tran} we also have:
\begin{Cor}
There exists a transitive hereditary shift with more than one ergodic measure of maximal entropy.
\end{Cor}

The above construction also can be modified in such a way that the obtained system has only one minimal subset. Choose a sequence of prime numbers $p_n\to\infty$. We will now define $x_A'$ by ``erasing'' some positions in $x_A$. Namely, whenever
$$
\mbox{$n = k-1 \bmod p_1\cdot\ldots\cdot p_k$ for some $k\geq 1$ and $n\neq k-1$},
$$
we put $x_A'(9n+i):=0$ for $0\leq i\leq 8$ (at all other positions the sequences $x_A$ and $x_A'$ are the same). We also define $x_B'$ adjusting in a similar way $x_B$. Let $X'_A$ and $X'_B$ be the closure of the orbit under the shift map of $x'_A$ and $x'_B$, respectively. Notice that arbitrarily long blocks of $0$'s occur on $x'_A$ and $x'_B$ with bounded gaps. Therefore, the singleton $\{(\ldots,0,0,\ldots)\}$ is the only minimal subset of $X:=X'_A\cup X'_B$. The same applies to $\widetilde{X}$, i.e.\ to the minimal hereditary subshift containing $X$. Notice that
$h_{top}(\widetilde{X})=h_{top}(\widetilde{X}'_A)=h_{top}(\widetilde{X}'_B)$.
Similar arguments as the ones used in Proposition~\ref{pr:1} show that the measures of maximal entropy on $\widetilde{X}'_A$ and $\widetilde{X}'_B$ are not the same. Moreover, in view of Lemma~\ref{lm:tran}, we can enlarge $\widetilde{X}$, so that it becomes transitive, remains hereditary and the topological entropy does not change. Finally, notice that the proof of Lemma~\ref{lm:tran} is carried out in such a way that whenever arbitrarily long blocks of $0$'s occur on $\widetilde{X}$ with bounded gaps, then the same is true for the enlarged system (see~\eqref{eq:A}). Therefore the singleton $\{(\ldots,0,0,\ldots)\}$ is the only minimal subset for the enlarged system.

\subsubsection{Uncountably many measures of maximal entropy}
For $y,a\in \mathbb{R}$ we define a sequence $x^{(y,a)}\in \{0,1\}^\Z$ in the following way:
$$
x^{(y,a)}(n):=\mathbbm{1}_{[0,1/2)}(\{y+na\}).
$$
We will write $x^{(a)}$ for $x^{(0,a)}$. Let $X_a:=\overline{\mathcal{O}(x^{(a)})}$.
Clearly, for any $a\in\R$, $h_{top}(X_a)=0$ and if $a\not\in\Q$ then $x^{(y,a)}\in X_a$ for any $y\in\R$.

Now, we choose an uncountable set $\mathcal{A}\subset\mathbb{R}\setminus\mathbb{Q}$ satisfying the following conditions:
\begin{itemize}
\item
any $\alpha\in \mathcal{A}$ has bounded partial quotients with $a_n(\alpha)\leq 2$,
\item
for any $\alpha, \beta\in \mathcal{A}$, the set $\{1,\alpha,\beta\}$ is rationally independent.
\end{itemize}
Let now $X:=\cup_{\alpha\in \mathcal{A}}X_\alpha$ and let $\widetilde{X}$ be the smallest hereditary subshift containing~$X$.

\begin{Remark}
It follows from Lemma~\ref{lm:7} that  $h_{top}(X_{\alpha})=1/2\log 2$.
\end{Remark}

We define $\mu_\alpha$ in the following way (cf.\ Section~\ref{outline}). $X_\alpha$ is an almost 1-1 extension of a rotation on the circle, i.e.\ it has only one invariant measure. Now, in each block we erase each $1$ with probability 1/2. This is the measure of maximal entropy (cf.\ Proposition~\ref{pSturmian}).

\begin{Lemma}\label{lm:10}
If $\alpha\in\mathcal{A}$ and $\beta$ is such that $|\alpha-\beta|<\frac{1}{48n^2}$ for some $n$ then all $n$-blocks occurring on $x^{(\beta)}$ occur on $x^{(\alpha)}$.
\end{Lemma}
\begin{proof}
Notice first that by the assumption that $\alpha\in\mathcal{A}$, for each $k\neq 0$ we have
$$
\|k\alpha\|\geq \frac{1}{2\cdot |\text{sup}_{n\in\N}\  a_n(\alpha)|\cdot k}\geq\frac{1}{6k}.
$$
Therefore, for $0\leq k<k'\leq n-1$ we have
$$
\|k\alpha-k'\alpha\|\geq \frac{1}{6|k-k'|}>\frac{1}{6n}
$$
and
\begin{equation}\label{half}
\|k\alpha-k'\alpha+1/2\|=\inf_{p\in\Z}\frac{|2(k-k')\alpha-2p+1|}{2}\geq \frac{1}{24|k-k'|}\geq \frac{1}{24n}.
\end{equation}
Fix $\ell\in\Z$ and let $m\in\Z$ be such that
$$
\|\ell \beta-m\alpha\|+n\| \beta-\alpha\|<\frac{1}{48n}
$$
and
\begin{equation}\label{may2}
\mathbbm{1}_{[0,1/2)}(\ell\beta+k_0\beta)=\mathbbm{1}_{[0,1/2)}(m\alpha+k_0\alpha),
\end{equation}
where $0\leq k_0<n-1$ satisfies
\begin{multline*}
\min \{\|\ell\beta+k_0\beta\|,\|\ell\beta+k_0\beta-1/2\|\}\\
=\min_{0\leq k\leq n-1}\min  \{\|\ell\beta+k\beta\|,\|\ell\beta+k\beta-1/2\|\}
\end{multline*}
(such $m\in\Z$ exists since the orbit of $0$ by the rotation by $\alpha$ is dense). Then, for $0\leq k\leq n-1$,
\begin{equation}\label{gg}
\|(\ell \beta+ k\beta)-(m\alpha +k\alpha)\|<\frac{1}{48n}.
\end{equation}

We claim that there exists at most one $0\leq k_1\leq n-1$ such that
\begin{equation}\label{ext}
\|m\alpha+k_1\alpha\|<\frac{1}{48n}\text{ or }\|m\alpha+k_1\alpha-1/2\|<\frac{1}{48n}.
\end{equation}
Suppose that~\eqref{ext} does not hold. There are several possibilities, all of which can be treated in the same way. We will show how to proceed in the case where
$$
\|m\alpha+k\alpha\|<\frac{1}{48n} \text{ and }\|m\alpha+k'\alpha-1/2\|<\frac{1}{48n}
$$
for some $0\leq k<k'\leq n-1$. It follows by~\eqref{half} that
$$
\frac{1}{24n} \leq \|(k-k')\alpha +1/2\|=\|m\alpha + k\alpha -m\alpha-k'\alpha +1/2\|<\frac{1}{48n}+\frac{1}{48n},
$$
which yields a contradiction. Therefore, using~\eqref{may2},~\eqref{gg} and~\eqref{ext}, we obtain
$$
\mathbbm{1}_{[0,1/2)}(\ell\beta+k\beta)=\mathbbm{1}_{[0,1/2)}(m\alpha+k\alpha)\text{ for $0\leq k\leq n-1$},
$$
which completes the proof.
\end{proof}

\begin{Lemma}\label{lm:2}
$h_{top}(X)=0$.
\end{Lemma}
\begin{proof}
For $n\geq 1$ fix $\alpha_1^{(n)},\dots, \alpha_{49n^2}^{(n)}\in \mathcal{A}$ such that for all $\alpha\in \mathcal{A}$ there exists $i$ such that $|\alpha-\alpha_i^{(n)}|<\frac{1}{48n^2}$. It follows from Lemma~\ref{lm:10} that the number of possible $n$-blocks is of order $n^3$ which ends the proof.
\end{proof}

\begin{Lemma}\label{lm:11}
For any $\vep>0$ there exists $n_0$ such that for $n\geq n_0$ the density of $1$'s in all $n$-blocks in $X$ is $\vep$-close to $1/2$.
\end{Lemma}
\begin{proof}
As the indicator function of the upper semicircle is Riemann integrable, we can approximate it by trigonometric polynomials, so that
\begin{equation}\label{eq:poly}
\sum_{-\tau_0}^{\tau_0}a_k e^{2\pi i kt}\leq \mathbbm{1}_{[0,1/2)}(t)\leq \sum_{-\tau_0}^{\tau_0}b_k e^{2\pi i k t},
\end{equation}
with $|a_0-1/2|, |b_0-1/2|<\delta$. Since all $\alpha\in\mathcal{A}$ have bounded partial quotients with $\text{sup}_{n\in\N}\ a_n(\alpha)\leq 2$, there exists $c>0$ such that for each $-\tau_0\leq k\leq \tau_0$, $k\neq 0$,
$$
\left|\frac{1}{n}\sum_{m=0}^{n-1}e^{2\pi i k (x+m\alpha)} \right|=\left|\frac{1}{n}\sum_{m=0}^{n-1}e^{2\pi i k m\alpha} \right|\leq \frac{1}{n}\frac{2}{|1-e^{2\pi i k\alpha}|}\leq \frac{1}{n}\cdot\frac{k}{c}.
$$
This, together with~\eqref{eq:poly}, completes the proof.
\end{proof}

\begin{Lemma}
$h_{top}(X)=1/2\log 2$.
\end{Lemma}
\begin{proof}
It suffices to apply Lemma~\ref{lm:2}, Lemma~\ref{lm:11} and Lemma~\ref{lm:7}.
\end{proof}

\begin{Lemma}
For $\alpha,\beta \in\mathcal{A}$, $\mu_\alpha\neq\mu_\beta$ if $\alpha\neq \beta$.
\end{Lemma}
\begin{proof}
Notice first that the closed support of $\mu_\alpha$ contains the minimal system $X_\alpha$. If $N$ is large enough then the orbit of any point $(x,y)$ will spend approximately 1/4 of time in the upper left quarter of $[0,1)\times [0,1)$ on its orbit of length $N$. The choice of $N$ is uniform, due to unique ergodicity of the rotation by $(\alpha,\beta)$ on $\mathbb{T}^2$ (recall that $\{1,\alpha,\beta\}$ are rationally independent).
This can be interpreted in the following way: for any block $B_\alpha$ in $X_\alpha$ and any block $B_\beta$ in $X_\beta$ at approximately half of the places where we can see a $1$ in $X_\alpha$, we see a $0$ in $X_\beta$. This however means that $B_\alpha$ cannot be seen on $\widetilde{X}_\beta$ and the claim follows as we have found a block of positive $\mu_\alpha$ measure and zero $\mu_\beta$ measure.
\end{proof}

\vspace{2ex}

\noindent Joanna Ku\l aga-Przymus:\\
Institute of Mathematics, Polish Academy of Sciences, \'{S}niadeckich 8, 00-956 Warsaw, Poland \\
and
\\
Faculty of Mathematics and Computer Science, Nicolaus Copernicus University, Chopina 12/18, 87-100 Toru\'{n}, Poland\\
\textit{E-mail address:} \texttt{joanna.kulaga@gmail.com}\\
\\
\noindent Mariusz Lema\'nczyk:\\
Faculty of Mathematics and Computer Science, Nicolaus Copernicus University, Chopina 12/18, 87-100 Toru\'{n}, Poland\\
\textit{E-mail address:} \texttt{mlem@mat.umk.pl}\\
\\
\noindent
Benjamin Weiss:\\
Institute of Mathematics, Hebrew University of Jerusalem, Jerusalem,
Israel\\
\textit{E-mail address:} \texttt{weiss@math.huji.ac.il}\\
\end{document}